\documentclass{article}
\usepackage{amsmath,amsthm,amssymb,amsfonts,amscd,eucal}
\usepackage[all]{xy}
\usepackage{graphicx}
\usepackage{epsfig}
\usepackage{epstopdf} 

\newtheorem{theorem}{Theorem}[section]
\newtheorem{proposition}[theorem]{Proposition}
\newtheorem{lemma}[theorem]{Lemma}
\newtheorem{corollary}[theorem]{Corollary}
\newtheorem{example}[theorem]{Example}

\theoremstyle{definition}
\newtheorem{definition}[theorem]{Definition}
\theoremstyle{remark}
\newtheorem{remark}[theorem]{Remark}

\DeclareFontFamily{OT1}{pzc}{}
\DeclareFontShape{OT1}{pzc}{m}{it}{<-> s * [1.10] pzcmi7t}{}
\DeclareMathAlphabet{\mathpzc}{OT1}{pzc}{m}{it}

\numberwithin{equation}{section}

\def\ca{{\mathcal A}}
\def\cb{{\mathcal B}}

\def\cd{{\mathcal D}}

\def\ch{{\mathcal H}}

\def\cl{{\mathcal L}}
\def\cam{{\mathcal M}}

\def\car{{\mathcal R}}

\def\cu{{\mathcal U}}

\def\cz{{\mathcal Z}}

\def\bc{{\mathbb C}}

\def\bn{{\mathbb N}}
\def\br{{\mathbb R}}
\def\bt{{\mathbb T}}
\def\bz{{\mathbb Z}}
\def\bno{\bn\cup\{0\}}

\def\a{\alpha}
\def\b{\beta}
\def\g{\gamma}        \def\G{\Gamma}
\def\d{\delta}        
\def\eps{\varepsilon}

\def\th{\vartheta}

\def\l{\lambda}       
\def\m{\mu}

\def\r{\rho}
\def\s{\sigma}        
     
\def\t{\tau}

\def\f{\varphi}
\def\c{\chi}
\def\o{\omega}        \def\O{\Omega}

\def\zcl{\mathpzc{l}} 
\def\zcp{\mathpzc{p}} 

\newcommand{\norm}[1]{\left\Vert#1\right\Vert}
\newcommand{\ceil}[1]{\lceil #1 \rceil}

\DeclareMathOperator{\diag}{diag}

\DeclareMathOperator{\Aut}{Aut}
\newcommand{\res}{\textrm{res}}
\newcommand{\ad}{\textrm{ad}}
\DeclareMathOperator{\tr}{tr}

\def\molt{r} 
\def\som{d} 

\begin{document}

\title{Spectral triples for noncommutative solenoidal spaces from self-coverings}
\author{
Valeriano Aiello$^\dag$, Daniele Guido$^\sharp$, Tommaso Isola$^\sharp$\footnote{The  authors were partially 
supported by MIUR and GNAMPA.  E-mail: 
valerianoaiello@gmail.com, guido@mat.uniroma2.it, isola@mat.uniroma2.it}
\\
$^\dag$ Dipartimento di Matematica e Fisica,\\ Universit\`a Roma Tre,\\ Largo S. Leonardo Murialdo 1, 00146 Roma, Italy.\\
$^\sharp$ Dipartimento di Matematica,\\ Universit\`a di Roma ``Tor
Vergata'',\\ I--00133 Roma, Italy.}

\date{\today}
\maketitle

\begin{abstract}
Examples of noncommutative self-coverings are described, and spectral triples on the base space are extended to spectral triples on the inductive family of coverings, in such a way that the covering projections are locally isometric. Such triples are shown to converge, in a suitable sense, to a semifinite spectral triple on the direct limit of the tower of coverings, which we call noncommutative solenoidal space. Some of the self-coverings described here are given by the inclusion of the fixed point algebra in a C$^*$-algebra acted upon by a finite abelian group. In all the examples treated here, the noncommutative solenoidal spaces have the same metric dimension and volume as  on the base space, but are not quantum compact metric spaces, namely the pseudo-metric induced by the spectral triple does not produce the weak$^*$ topology on the state space. 

\end{abstract}

\tableofcontents

\section{Introduction}

Given a noncommutative self-covering consisting of a C$^*$-algebra with a unital injective endomorphism $(\ca,\a)$, we study the possibility of extending a spectral triple on $\ca$ to a spectral triple on the inductive limit $C^*$-algebra, where,  as in \cite{Cuntz}, the inductive family associated with the endomorphism $\a$ is
\begin{align}\label{ind-lim-diag}
\ca_0\stackrel{\a}{\longrightarrow}\ca_1\stackrel{\a}{\longrightarrow}\ca_2\stackrel{\a}{\longrightarrow}\ca_3\dots,
\end{align}
all the $\ca_n$ being isomorphic to $\ca$. The algebra $\ca_n$ may be considered as the $n$-th covering of the algebra $\ca_0$ w.r.t. the endomorphism $\a$.
As a remarkable byproduct, all the spectral triples we construct on the inductive limit $C^*$-algebra are semifinite spectral triples. 

Let us recall that the first notion of type II noncommutative geometry appeared in \cite{CoCu}, where semifinite Fredholm modules were introduced, a notion then  generalized in \cite{CaPhi1}, see also \cite{CaPhi2}, with that of semifinite unbounded Fredholm module. The latter is essentially the same definition as that of von Neumann spectral triples of \cite{BeFa}, where
some previous constructions \cite{Atiy,Conn,Shub,Lesc} were reinterpreted as examples of such concept. In the same period, \cite{PaRe} considered semifinite spectral triples for graph algebras and posed the problem of exhibiting more examples of the kind, which was done in \cite{PaReSi} using $k$-graph algebras  and in \cite{AGNe} inspired by quantum gravity. 
Further examples have been considered in \cite{CaGaReSu,Wahl}.

In the cases we analyze, it is possible to construct natural spectral triples on the C$^*$-algebras $\ca_n$ of the inductive family, which converge,  in a suitable sense, to a triple on the inductive limit, and the latter triple is indeed semifinite.

The leading idea is that of producing geometries on each of the noncommutative coverings $\ca_n$ which are locally isomorphic to the geometry on the original noncommutative space $\ca$.  This means in particular that the covering projections should be local isometries or, in algebraic terms, that the noncommutative metrics given by the Lip-norms associated with the Dirac operators via $L_n(a)=\|[D_n,a]\|$ (cf. \cite{ConnesBook,RieffelMetrics}) should be compatible with the inductive maps, {\it i.e.}
$L_{n+1}(\a(a))=L_n(a)$, $ a\in\ca_n$. In one case, this property will be weakened to the existence of a finite limit for the sequences $L_{n+p}(\a^p(a))$, $ a\in\ca_n$.

The above request  produces two related effects. On the one hand, the noncommutative coverings are metrically larger and larger, so that their radii diverge to infinity, and {\it the inductive limit is topologically compact} (the C$^*$-algebra has a unit) {\it but not totally bounded} (the metric on the state space does not induce the weak$^*$-topology). On the other hand, the spectrum of the Dirac operator becomes more and more dense in the real line, so that the resolvent of the limiting Dirac operator cannot be compact, being indeed $\tau$-compact w.r.t. a suitable trace, and thus producing a {\it semifinite spectral triple on the inductive limit}.

Pursuing this idea means also that we see the elements of the  inductive family in a more geometric way, namely as distinct (though isomorphic) algebras of \textquotedblleft functions" on noncommutative coverings, and the inductive maps as embeddings of a sub-algebra into an algebra of \textquotedblleft less periodic" functions. In this sense the inductive limit is a noncommutative version of the solenoidal spaces in \cite{McCord}, see also the noncommutative solenoids in \cite{LP2,LP}.

The first four sections are devoted to the study of {\it noncommutative regular (self-)coverings with finite abelian group}, namely in particular of a C$^*$-algebra $\ca_1$ acted upon by a finite abelian group $\Gamma$ whose fixed point algebra $\ca_0$ is isomorphic to $\ca_1$. A further property, which we call regularity, requires that the eigenspaces of $\ca_1$ w.r.t. the action of $\Gamma$ contain invertible elements. This turns out to imply that $\ca_1$ can be seen as a subalgebra of a matrix algebra on $\ca_0$, and in this way the algebras $\ca_n$ forming the inductive family described above are naturally embedded into $\ca_0\otimes M_r(\bc)^{\otimes n}$, $r=|\Gamma|$. The resulting embedding of the inductive limit into $\ca_0\otimes \mathrm{UHF}(r^\infty)$ provides the semifinite environment for the spectral triple on the inductive limit.
The regularity assumption  also implies that the ``can" map is an isomorphism (cf. Section 1.4), namely the regularity property of the covering according to \cite{BDH}.

In Sections 2, 3 and 4 we study  coverings of the torus and generalizations to noncommutative tori and to crossed products with $\bz^n$, based on a non-degenerate  integer-valued matrix $B$. This implies in particular that the regularity property holds and the invertible elements in the eigenspaces of the action $\Gamma$ can be chosen in terms of a section $s$ of the exact sequence 
\begin{equation}\label{exactseq}
0\to \bz^p\to(B^T)^{-1}\bz^p\to (B^T)^{-1}\bz^p/\bz^p\to 0,
\end{equation}
 cf. equation \eqref{map-sigma}. The choice of a particular section  plays no role in the definition of the Dirac operator and of the Lip-norm on the $n$-th covering quantum spaces, however such section enters  the formulas for the identification of the covering algebras as algebras of matrices on the base algebra. It turns out that  choosing ``minimal'' sections and requiring the matrix $B$ to be purely expanding guarantees a suitable convergence of the Dirac operators on the $n$-th covering to a Dirac operator on the inductive limit.
 
As mentioned above, the first example of regular covering is described in Section 2, and is indeed a classical covering, namely the self-covering of the $p$-torus $\br^p/\bz^p$ given by a non-degenerate matrix $B\in M_p(\bz)$. We assume $\det B\ne\pm1$ to avoid the automorphism case. The covering map is the projection $\br^p/B\bz^p\longrightarrow\br^p/\bz^p$, the group of deck transformations being $\Gamma=\bz^p/B\bz^p$. The corresponding embedding for the algebras is  $C(\br^p/\bz^p)\hookrightarrow C(\br^p/B\bz^p)$, the group $\bz^p/B\bz^p$ acts on the larger algebra having the smaller as fixed point algebra. As mentioned before, the algebras $\ca_n$, consisting of continuous functions on the $n$-th covering, can be represented as matrices on $\ca_0$, namely embed into $\ca_0\otimes M_r(\bc)^{\otimes n}$. 
Endowing the $n$-th covering  with the pullback of the metric on the base space makes the covering projections locally isometric. The corresponding Dirac operator on $\ca_n$  is formally identical to that on $\ca_0$; when $\ca_n$ is described as a sub-algebra of $\ca_0\otimes M_r(\bc)^{\otimes n}$, the Dirac operator $D_n$ is affiliated to  $\cb(\ch)\otimes M_r(\bc)^{\otimes n}$, $\ch$ being the Hilbert space of the spectral triple on $\ca_0$. When $n\to\infty$, $\ca_\infty\subset\ca_0\otimes \mathrm{UHF}(r^\infty)\subset \cb(\ch)\otimes \car$, where $\car$ is the unique injective type II$_1$ factor obtained as the weak closure of the UHF algebra in the GNS representation of the  unital trace. Moreover, with a suitable choice of the section $s$ in \eqref{exactseq} and under the assumption that the matrix $B$ is purely expanding, the sequence $D_n$ converges  to an operator $D_\infty$ affiliated with $\cb(\ch)\otimes \car$. The triple $(\ca_\infty, \cb(\ch)\otimes \car,D_\infty)$ turns out to be a semifinite spectral triple and the metric dimension, given by the abscissa of convergence $d_\infty$ of the zeta function $\tau((1+D_\infty^2)^{s/2})$, and the noncommutative volume, given by the residue in $d_\infty$ of the zeta function, coincide with the corresponding quantities for the base torus.

Section 3 contains the extension of the results for the torus to the case of rational noncommutative 2-tori $A_\vartheta$, $\vartheta=p/q$.  In this case we get a self-covering if $\det B\equiv1$ mod $q$. With this proviso, for $B$  purely expanding, we get a semifinite spectral triple on the inductive limit, and the metric dimension and the noncommutative volume are the same as those of the base torus.

The third example of noncommutative regular self-coverings with finite abelian group is treated in Section 4, where we consider the covering $\cz\rtimes_\rho \bz^p\hookrightarrow\cz\rtimes_\rho (B^T)^{-1}\bz^p$, $B\in M_p(\bz)$ with $\det B \not\in \{0,\pm 1\}$ as above.
The action of $\bz^p/B\bz^p$ by translation on $C(\br^p/B\bz^p)$ induces an action  on  $\cz\rtimes_\rho(B^T)^{-1}\bz^p$ whose fixed point algebra is $\cz\rtimes_\rho \bz^p$.
We get a self-covering if the actions $g\in\bz^p\to \rho_g\in \Aut(\cz)$ and $g\in\bz^p\to \rho_{(B^T)^{-1}g}\in \Aut(\cz)$ are conjugate by an automorphism $\beta$ of $\cz$. Under this assumption the C$^*$-algebras $\ca_n$ are all isomorphic and may be described as  $\cz\rtimes (B^T)^{-n}\bz^p$, with the action defined in terms of the action of $\bz^p$ and the automorphism $\beta$.

The study of the extension of a spectral triple on a C$^*$-algebra to a triple on  crossed products was initiated in \cite{BMR} and then pursued in \cite{Skalski,Paterson}.
Such extension requires a choice of a (proper, translation bounded, matrix-valued) length function on the group.
We assume $\cz$ is endowed with a spectral triple, and use the mentioned results to extend such triple to the algebras $\ca_n$, with a suitable choice of a length function on $(B^T)^{-n}\bz^p$. However a further assumption is required in order to prove that the covering projection is locally isometric. The same results as in the previous sections hold for the spectral triple on the inductive limit $C^*$-algebra.

The subsequent section contains the analysis of an example of covering which is not given by an action of a finite abelian group. It consists of a UHF algebra with the shift endomorphism, the spectral triple being the one described in \cite{Chris}. In this case the C$^*$-algebras $\ca_n$ are naturally given by tensoring $\ca_0$ with $\cam(\bc)^{\otimes n}$, and we choose the Dirac operators to have the same form as that on $\ca_0$. However in this way local isometricity is not exactly satisfied,  while, as mentioned above,  the sequence $L_{n+p}(\a^p(a))$ converges when $n\to\infty$, for $ a$ in a suitable dense sub-algebra of $\ca_n$. As in the previous sections, we obtain a semifinite spectral triple on the inductive limit C$^*$-algebra with the same metric dimension and  noncommutative volume of the base space.

Section 6 deals with the metric properties induced by the limit spectral triple on the inductive limit C$^*$-algebra. We show that, in all the examples considered, the radii of the algebras $\ca_n$ diverge, giving rise to a non totally bounded noncommutative space. As explained before, this is a consequence of the fact that the covering projections are  locally isometric. 

In the second part of the section we compare our results with those in \cite{LP}, where the direct limits of noncommutative tori, called noncommutative solenoids, are seen as twisted group C$^*$-algebras acted upon by the solenoid group. 
It is shown there that the inductive sequence of noncommutative tori converges in the quantum Gromov-Hausdorff metric (and in the Gromov-Hausdorff propinquity) to the noncommutative solenoid, when the latter is endowed with a Lip-norm given by a suitable choice of the length function on the solenoid group (cf. \cite{Rief}). 
It turns out that our seminorm on the inductive limit is also induced, {\it a la} Rieffel, by a length function on the solenoid group, but our function is infinite on some elements, thus giving rise to a non totally bounded space. It would be interesting to know if also our sequence of quantum compact  metric spaces tends to the direct limit  w.r.t. some kind of (pointed) quantum Gromov-Hausdorff convergence.

We mention in conclusion that a motivation for the study of spectral triples for direct limits is the attempt to extend the constructions in \cite{BMR,Skalski,Paterson} to the case of a crossed product by a single endomorphism, that is, to the case of self-covering. The corresponding results are contained in \cite{AGI02}. An example of  self-covering with ramification and the study of the corresponding crossed product is contained in a further paper \cite{AGI03}.

\section{Noncommutative coverings w.r.t. a finite abelian  group} \label{NC-coverings}
\subsection{Spectral decomposition}
The aim of this section is to describe a spectral decomposition of an algebra in terms of an action of a finite abelian group. For more details and a general theory the interested reader is referred to \cite{Ped}. 

Let $\cb$ be a $C^*$-algebra and $\Gamma$ be a finite abelian group which acts on $\cb$ (we denote the action by $\g$). Let
\begin{displaymath}
	\cb_k:=\{b\in \cb \textrm{ s.t. } \g_g(b)=\langle k,g \rangle b \quad \forall g\in \Gamma \}, \quad  k\in \widehat{\Gamma}.
\end{displaymath}

\begin{proposition} \label{prop-1.1}
With the above notation,
\begin{enumerate}
\item[$(1)$] $\cb_h\cb_k\subset \cb_{hk}$; in particular each $\cb_k$ is an $\ca$-bimodule, where $\ca$ is the fixed point subalgebra,

\item[$(2)$] if $b_k\in \cb_k$ is invertible, then $b_k^{-1},b_k^*\in \cb_{k^{-1}}$,

\item[$(3)$] each $b\in \cb$ may be written as $\sum_{k\in \widehat{\Gamma}} b_k$ with $b_k\in \cb_k$.
\end{enumerate}
\end{proposition}
Before proving this proposition we recall that by the
Schur orthogonality relations, \cite{Serre},
given $\Gamma$ a finite abelian group, $\widehat{\Gamma}$ its dual,
\begin{eqnarray}\label{schur-ort-rel}
\sum_{k\in\widehat{\Gamma}}\langle k,g\rangle &= &\d_{g,e}\cdot |\Gamma | \qquad \forall g\in \Gamma .
\end{eqnarray}

\begin{proof} The first two properties follow by definition.
Let us set 
$$
b_k \equiv E_k(b) \stackrel{def}{=}\frac{1}{|\Gamma|}\sum_{g\in\Gamma}\langle k^{-1},g\rangle \g_g(b).
$$
Then, by \eqref{schur-ort-rel},
\begin{eqnarray*}
	\sum_{k\in \widehat{\Gamma}} b_k &=& \frac{1}{|\Gamma|} \sum_k\sum_g \langle k^{-1},g\rangle \g_g(b) \\
	&=&  \frac{1}{|\Gamma|} |\Gamma|\sum_g \delta_{g,e} \g_g(b)=b.
\end{eqnarray*}
Finally, $b_k$ belongs to $\cb_k$ since, for any $g\in\Gamma$,
\begin{eqnarray*}
	\g_{g}(b_k) & = & \frac{1}{|\Gamma|}\sum_{h\in\Gamma} \langle k^{-1},h \rangle \g_{g}\g_h(b) \\
	&=&\langle k^{-1},g^{-1}\rangle\frac{1}{|\Gamma|}\sum_{h\in\Gamma}\langle k^{-1},h\rangle \g_h(b) \\
	&=&\langle k,g\rangle b_k.
\end{eqnarray*}
\end{proof}

\subsection{Noncommutative coverings} \label{section-nc-covering}

\begin{definition} \label{def-reg-cov}
	A \emph{finite (noncommutative) covering} with abelian group is an inclusion of  (unital) $C^*$-algebras $\ca\subset \cb$ together with an action of a finite abelian group $\Gamma$ on $\cb$ such that $\ca=\cb^\Gamma$. We will say that $\cb$ is a covering of $\ca$ with deck transformations given by the group $\Gamma$.
\end{definition}

Let us denote by $M_{\widehat{\Gamma}}(\cb)$ the algebra of matrices, whose entries belong to $\cb$ and are indexed by elements of $\widehat{\Gamma}$. Then, to any $b\in \cb$, we can associate  the matrix $\widetilde{M}(b)\in M_{\widehat{\Gamma}}(\cb)$ with the following entries
\begin{displaymath}
	\widetilde{M}(b)_{hk}=b_{h-k},\quad h,k\in\widehat{\Gamma}.
\end{displaymath}

By the definition of $b_k$ the following formula easily follows
\begin{eqnarray} \label{Mtilde}
		\widetilde{M}(b)\widetilde{M}(b')=\widetilde{M}(bb').
\end{eqnarray}
The following definition is motivated by Theorem \ref{caniso} below.
\begin{definition}\label{unitaries}
We say that the finite covering $\ca\subset\cb$ w.r.t. $\Gamma$ is regular if each $\cb_k$ has an element which is unitary in $\cb$, namely we may choose a map $\sigma: \widehat{\Gamma}\to \cb$ such that $\sigma(k)\in U(\cb)\cap \cb_k$, with $\sigma(e)=I$.
\end{definition}

\begin{remark}\label{regcov}
\item{$(i)$} Example \ref{nounitaries} shows this assumption does not always hold.
\item{$(ii)$} In the previous definition, it is enough to ask that each $\cb_k$ has an element which is invertible in $\cb$. Indeed, if $C\in\cb_k$ is invertible, and $C=UH$ is its polar decomposition, then $H=(C^*C)^{1/2}\in\ca$. It follows that  $U$ is unitary and belong to $\cb_k$.
\item{$(iii)$} Regularity also implies that the action is faithful. Indeed, if $g\in\Gamma$ acts trivially, we may find $k\in\widehat \Gamma$ such that $\langle k,g\rangle\ne1$, therefore the equation $\g_g(b)=\langle k,g \rangle b$ is satisfied only for $b=0$, and $\cb_k$ does not contain invertible elements.
\end{remark}

For regular coverings, we can define an embedding of $\cb$ into $M_{\widehat{\Gamma}}(\ca)$. Set 
\begin{displaymath}
	M(b)_{hk}=\sigma(h)^{-1}\widetilde{M}(b)_{hk}\sigma(k)
	=\sigma(h)^{-1}b_{h-k}\sigma(k), \quad h,k\in\widehat{\Gamma}.
\end{displaymath}
It follows from Proposition \ref{prop-1.1} that $M(b)_{hk}\in \ca$.

\begin{theorem}\label{embedding}
	Under the regularity hypothesis, the algebra $\cb$ is isomorphic to a subalgebra of matrices with coefficients in $\ca$, i.e. we have an embedding 
	\begin{equation}\label{eq:embedding}
	\cb\hookrightarrow \ca\otimes M_{\widehat{\Gamma}}(\mathbb C).
	\end{equation}
\end{theorem}
\begin{proof}
It is easy to show that $M(b^*)_{jk}=(M(b)_{kj})^*$, $\forall b\in\cb$, $j,k\in \widehat{\Gamma}$. That the product is preserved, namely 	
	\begin{displaymath}
		M(bb')_{hk}=\sum_j M(b)_{hj}M(b')_{jk} 
	\end{displaymath}
follows  easily from  \eqref{Mtilde}.
\end{proof}

In this paper we are mainly interested in self-coverings, namely when there exists an isomorphism $\phi:\cb\to \ca$ or, equivalently, $\ca$ is the image of $\cb$ under a unital endomorphism $\alpha=j\circ\phi$, where $j$ is the embedding of $\ca$ in $\cb$.

\begin{theorem}\label{embedding2}
Given a (noncommutative) regular self-covering with abelian group $\Gamma$, we may construct an inductive family $\ca_i$ associated with the endomorphism $\alpha$ as in \cite{Cuntz}. 
Then, setting $r=|\widehat{\Gamma}|=|\Gamma|$, we have the following embedding:
	$$
		\varinjlim \ca_i\hookrightarrow \ca\otimes UHF(r^{\infty}).
	$$
\end{theorem}  
\begin{proof}
By applying Theorem \ref{embedding} $j$ times, we get an embedding of $\ca_j$ into $\ca\otimes M_r^{\otimes j}$. The result immediately follows.
\end{proof}

The following example shows that  the regularity property in Definition \ref{unitaries} is not always satisfied.

\begin{example}\label{nounitaries}
	Let $\cb=M_3(\mathbb{C})$, $\Gamma=\mathbb{Z}_2=\{0,1\}$.  We have the following action $\gamma$ on $\cb$: $\gamma_0=id$, $\gamma_1=ad(J)$, where
	\begin{displaymath}
		J=\left(
 		\begin{array}{ccc}
 		1 & 0 & 0 \\ 0 & -1 & 0 \\ 0 & 0 & -1
 		\end{array}
 	\right).
	\end{displaymath}
	Therefore
	\begin{displaymath}
		\cb_0=\ca=\cb^\Gamma =\left\{ x\in \cb  :  x=\left(
 		\begin{array}{ccc}
 		a & 0 & 0 \\ 0 & b & c \\ 0 & d & e
 		\end{array}
 	\right) \right\}, \quad \cb_{1}=\left\{ x\in \cb  : x=\left(
 		\begin{array}{ccc}
 		0 & a & b \\ c & 0 & 0 \\ d & 0 & 0
 		\end{array}
 	\right) \right\}.
	\end{displaymath}
	Hence $\cb_1$ has no invertible elements.
\end{example}

\subsection{Representations}

\begin{proposition} \label{Representation-covering}
Consider a (noncommutative) regular self-covering $\ca\subset \cb$ with abelian group $\Gamma$.

\item[$(1)$] A representation $\pi$ of $\ca$ on a Hilbert space $H$ produces a representation $\widetilde{\pi}$ of $\cb$ on $H\otimes \bc^r$, $r=|\widehat{\Gamma}|$, given by $\widetilde{\pi}(b):=[ \pi(M(b)_{hk}) ]_{h,k\in\widehat{\Gamma}} \in M_{\widehat{\Gamma}} (\cb(H)) = \cb(H\otimes \bc^r)$, $\forall b\in\cb$. 

\item[$(2)$] If the representation of $\ca$ is induced by a state $\f$ via the GNS mechanism, the corresponding representation of $\cb$ on $H\otimes \bc^r$ is a GNS representation induced by the state $\tilde\f$, where $\tilde\f(b)=\f\circ E_\Gamma$, and $E_\Gamma$ is the conditional expectation from $\cb$ to $\ca$. Moreover, the map
\begin{equation}\label{isometry}
\begin{matrix}
\cb&\rightarrow&\ca\otimes\bc^r\\
b&\mapsto&(a_j)_{j\in\hat\Gamma}
\end{matrix}
\ ,\quad a_j=\sigma(j)^{-1}b_j
\end{equation}
extends to an isomorphism of the Hilbert spaces $L^2(\cb,\tilde\f)$ and $L^2(\ca,\f)\otimes\bc^r$.
\end{proposition}
\begin{proof}
$(1)$ It is a simple computation.

\item[$(2)$] Denoting by $\xi_\f$ the GNS vector in $H$, we set $\tilde{\xi_\f}$ to be the vector $\xi_\f$ in $H_e$ and 0 in the other summands.
It is cyclic, because
$$
\widetilde{\pi}(b) \tilde{\xi}_\f = \oplus_{k\in\widehat\Gamma} \ \sigma(k)^{-1}b_{k}\xi_\f.
$$
Since $\xi_\f$ is cyclic for $\ca$, $\{\sigma(k)^{-1}b_{k}\xi_\f : b_k\in \cb_k\}$ is dense in $H$.
It induces the state $\tilde\f$, since
$$
\big(\tilde{\xi}_\f,\widetilde{\pi}(b)\tilde{\xi}_\f\big)= \big(\xi_\f,b_{e}\xi_\f\big) = \f\left(\frac{1}{|\Gamma|}\sum_{g\in\Gamma} \g_g(b)\right)=\f\circ E_\Gamma(b).
$$
The isomorphism in \eqref{isometry} follows by the GNS theorem.
\end{proof}

\subsection{Finite regular coverings}
In this subsection we discuss the relation between our definition of  (noncommutative)  finite regular covering and the classical notion of regular covering. As a byproduct of an analysis on actions of compact quantum groups, it was proved in \cite{BDH} that a finite covering is regular iff the ``can'' map is an isomorphism. More precisely, if $X$ and $Y$ are compact Hausdorff  spaces and $\pi: X\to Y$ is a covering map    with finite group of deck transformations $\Gamma$,  $X$ is a regular covering of $Y$ if and only if the canonical map 
\begin{gather*}
{\rm can}: C(X)\otimes_{C(Y)} C(X)\to C(X)\otimes C(\Gamma)\\
f_1\otimes f_2 \to (f_1\otimes 1)\delta(f_2), 
\end{gather*}
is an isomorphism of C$^*$-algebras, where $\delta f = \sum_{g\in \Gamma} \g_{g^{-1}}(f)\otimes \chi_g$, $\g_g: \Gamma\to Aut(C(X))$ and $\chi_g$ denote the action induced by $\Gamma$ and the characteristic function on elements of $\Gamma$, respectively. 

The map ${\rm can}$  for classical coverings makes perfect sense in our case too
$$
{\rm can}: \cb\otimes_\ca \cb\to \cb\otimes C(\Gamma),
$$
where ${\rm can}(x\otimes y)=(x\otimes 1)\delta(y)$ and $\delta (y)=\sum_{g\in \Gamma} \g_{g^{-1}}(y) \otimes \chi_g$.  In our framework, however, the canonical map is no longer a morphism of C$^*$-algebras, it is a morphism of $(\cb-\ca)$-bimodules. In fact, this map clearly commutes with the left action of $\cb$. Moreover, the right action $\ca$ commutes with  ${\rm can}$ since $\delta|_\ca=id$. The following theorem shows that, under the regularity property of Definition \ref{unitaries}, the can map is an isomorphism, that is, the regularity property according to \cite{BDH}. 

\begin{theorem}\label{caniso}
Under the above hypotheses, the map ${\rm can}: \cb\otimes_\ca \cb\to \cb\otimes C(\Gamma)$ is an isomorphism of $(\cb-\ca)$-bimodules.
\end{theorem}
\begin{proof}
The group $\G\times \G$ clearly acts on $\cb\otimes_{\ca}\cb$, the eigenspaces being 
$(\cb\otimes_\ca \cb)_{j,k} = \{ \s(j)a\otimes\s(k) : a\in\ca \}$, $(j,k)\in\widehat\G\times \widehat\G$.
Therefore the elements of $\cb\otimes_{\ca}\cb$ can be written as 
$$
z=\sum_{j,k\in \widehat{\Gamma}} \s(j) a_{j,k}\otimes \s(k) \quad a_{j,k}\in \ca.
$$
Suppose that ${\rm can}(z)=0$. We want to prove that $z=0$.
Using the fact that $\cb$ is the direct sum of its eigenspaces we get
\begin{align*}
{\rm can}(z) & = \sum_{g\in \Gamma}\sum_{j,k\in \widehat{\Gamma}} \langle g^{-1}, k\rangle \s(j) a_{j,k}\s(k)\otimes \chi_g=0 \\
\Rightarrow &  \sum_{j,k\in \widehat{\Gamma}} \langle g^{-1}, k\rangle \s(j) a_{j,k}\s(k)=0 \quad \forall g\in \Gamma,
\end{align*}
where $a_{j,k}\in \ca$. Now we show that any $a_{j,k}$ is zero. In fact,
multiplying by $\langle g, \ell \rangle$ and summing over $g\in\Gamma$, we get  
\begin{align*}
0 = \sum_{g\in \Gamma} \langle g, \ell\rangle \sum_{j,k\in \widehat{\Gamma}} \langle g^{-1}, k\rangle \s(j) a_{j,k}\s(k) 
 = \sum_{j,k\in \widehat{\Gamma}} \langle g, \ell k^{-1}\rangle \s(j) a_{j,k}\s(k) 
 = |\Gamma| \sum_{j\in\widehat{\Gamma}} \s(j) a_{j,k}\s(k),
\end{align*}
which implies that $a_{j,k}=0$ for all $j,k \in \widehat{\Gamma}$, so that $z=0$.

Consider $\sum_{g\in \Gamma} b(g) \otimes \chi_g$, we want to show that it can be obtained as ${\rm can}(z)$ for some $z\in \cb\otimes_\ca \cb$.  By the above computations, it suffices to solve the following equation, for any $\ell\in \widehat{\Gamma}$,
$$
\sum_{g\in \Gamma} \langle g, \ell\rangle b(g) =\sum_{g\in \Gamma} \langle g, \ell\rangle \sum_{j,k\in \widehat{\Gamma}} \langle g^{-1}, k\rangle \s(j) a_{j,k}\s(k),\\
$$
which, using \eqref{schur-ort-rel}, may be rewritten as 
$$
\sum_{g\in \Gamma} \langle g, \ell \rangle b(g) \s(k)^{-1} = |\Gamma|  \sum_{j \in \widehat{\Gamma}}   \s(j) a_{j,k}.
$$
Since each $b(g)$ is given, the cofficients $a_{j,k}$ can be uniquely determined using again the decomposition of $\cb$ in its eigenspaces.
\end{proof}

\section{Self-coverings of tori}\label{Tori-cov}

\subsection{The $C^*$-algebra and its spectral triple}
We consider the $p$-torus $\bt^p=\mathbb{R}^p/\mathbb{Z}^p$ endowed with the usual metric, inherited from $\mathbb{R}^p$. On this Riemannian manifold we have the Levi-Civita connection $\nabla^{LC}=d$ and we can define the Dirac operator acting on the Hilbert space $\bc^{2^{[p/2]}} \otimes L^2(\bt^p,dm)$ 
\[
D=-i\sum_{a=1}^p \eps^a \otimes \partial^a,
\]
where $\eps^a= (\eps^a)^*\in M_{2^{[p/2]}}(\bc)$   
furnish a representation of the Clifford algebra for the $p$-torus (see \cite{Spin} for more details on Dirac operators).
Therefore, we have the following spectral triple 
\begin{eqnarray*}
	(C^1(\bt^p),  \bc^{2^{[p/2]}} \otimes L^2(\bt^p,dm), D=-i\sum_{a=1}^p \eps^a\otimes  \partial^a).
\end{eqnarray*}

\subsection{The covering}
Consider an integer-valued matrix $B\in M_p(\bz)$ with $|\!\det (B)|=:r>1$. This defines a covering of $\bt^p$ as follows. Let us set $\bt_1=\br^p/B\bz^p$ seen as a covering space of $\bt_0:=\bt^p$. Clearly $\bz^p$ acts on $\bt_1$ by translations, the subgroup $B\bz^p$ acting trivially by definition, namely we have an action of $\bz_B:=\bz/B\bz^p$ on $\bt_1$, which is simply the group of deck trasformations for the covering. We denote this action by $\gamma$.
We are now in the situation described in the previous section, with $\ca=C(\bt_0)$ the fixed point algebra of $\cb=C(\bt_1)$ under the action of $\bz_B$.  These algebras can be endowed with the following states, respectively
\begin{eqnarray*}
\tau_0(f) &=& \int_{\bt_0} f dm, \quad f\in \ca,\\
\tau_1(f) &=& \frac{1}{|\!\det (B)|}\int_{\bt_1} f dm, \quad f\in \cb,
\end{eqnarray*}
where $dm$ is Haar measure.

\begin{proposition}
The GNS representation  $\pi_1:\cb\to B(L^2(\cb, \tau_1))=B(L^2(\bt_1, dm))$ is unitarily equivalent to the representation $\widetilde{\pi}_0$ obtained by $\pi_0: \ca\to B(L^2(\ca, \tau_0))=B(L^2(\bt_0, dm))$ according to Proposition \ref{Representation-covering}.
\end{proposition}
\begin{proof}
By the GNS theorem it is enough to check that  $\tau_1=\tau_0\circ E$, where $E$ denotes the conditional expectation from $\cb$ to $\ca$. This follows from the following observation on the associated measures: they are both probability measures that are traslation invariant, by the results on  Haar measures the claim follows.  
\end{proof}

In order to apply the results of the previous section, we need to choose unitaries in the eigenspaces $\cb_k$, $k\in\widehat{\bz_B}$, namely a map $\sigma:g\in\widehat{\bz_B}\to\cu(\cb)\cap \cb_g$.

\bigskip

With $\bt_0=\br^p/\bz^p$, $\bt_1=\br^p/B\bz^p$, $\bz_B=\bz^p/B\bz^p$ as above,  set $A=(B^T)^{-1}$, $\langle x,y\rangle = \exp\big( 2\pi i\sum_{a=1}^p x^ay^a \big)$, $x,y\in\br^p$. 
 
\begin{lemma}
With the above notation

\item[$(1)$] The cardinality $|\bz_B|$ of $\bz_B$ is equal to $\molt$,

\item[$(2)$] the following duality relations hold: $\widehat{\bt}_0=(\br^p/\bz^p)^{\widehat{}}=\bz^p$,
$\widehat{\bt}_1=(\br^p/B\bz^p)^{\widehat{}}=A\bz^p$, 
$\widehat{\bz_B}=(\bz^p/B\bz^p)^{\widehat{}}=A\bz^p/\bz^p$.

In particular, the duality $\langle z,g\rangle$, $g\in \bt_1,z\in A\bz^p$ induces the duality $\langle k,g \rangle_o$, $g\in\bz_B,k\in\widehat{\bz_B}$, namely if $g\in\bz_B\subset\bt_1$,  
$\langle z,g\rangle=\langle \dot{z},g\rangle_o$, where $\dot{z}$ denotes the class of $z$ in $\widehat{\bz_B}$. For this reason we drop the subscript $o$ in the following.
\end{lemma}
\begin{proof}
The proofs of the claims are all elementary. We only make some comments on the first one. It is well known that each finite abelian group is  the  direct sum of  cyclic groups and that the order of these groups can be obtained with the following procedure.   Let $D=SBT$ the Smith normal form of $B$, where $S, T\in GL(p, \bz)$ and $D=\diag(d_1,\cdots d_p)>0$. 
Therefore, we have that $\bz_B=\bz^p/B\bz^p\cong \bz^p/D\bz^p$. As $B$ is invertible, so is $D$ and all the diagonal elements are non-zero. Thus, $\bz_B=\bz_{d_1}\oplus\ldots \oplus \bz_{d_p}$ and $|\bz_B|=d_1\cdot \ldots \cdot d_p=\det(D)=\pm\det(B)$.
\end{proof}

Let us consider the short exact sequence of groups
\begin{align}
& 0\longrightarrow \bz^p\longrightarrow A\bz^p\longrightarrow\widehat{\bz_B}\longrightarrow 0. \label{seq0}
\end{align}
Such central extension $A\bz^p$ of $\widehat{\bz_B}$ via $\bz^p$ can be described either with a section $s:\widehat{\bz_B}\to A\bz^p$ or via a $\bz^p$-valued 2-cocycle $\o(k,k')=s(k)+s(k')-s(k+k')$, see e.g. \cite{Brown}.
We choose the unique section such that, for any $k\in \widehat{\bz_B}$, $s(k)\in
 [0,1)^p$. 

\begin{remark}
The mentioned choice of the section $s$ will play a role only later. For the moment, we only note that it implies $s(0)=0$, hence $\o(k,0)=0=\o(0,k)$.
\end{remark}

The covering we are studying is indeed regular according to Definition  \ref{unitaries}, since we may  construct the map $\sigma$ as follows:
\begin{equation}\label{map-sigma}
\s(k)(t):=\overline{\langle s(k),t\rangle}, \quad k\in\widehat{\bz_B}, t\in\bt_1.
\end{equation}

\subsection{Spectral triples on covering spaces of $\bt^p$} \label{ST-CS}

Given the integer-valued matrix $B\in M_p(\bz)$ as above, if $\bt^p$ is  identified with $\br^p/\bz^p$, then there is an associated self-covering $\pi:t\in\bt^p\mapsto Bt\in\bt^p$. We denote by $\a$ the induced endomorphism of $C(\bt^p)$, i.e. $\a(f)(t)=f(Bt)$. Then we consider the inductive limit $\ca_\infty=\displaystyle \varinjlim \ca_n$ described in \eqref{ind-lim-diag}, where $\ca_n=\ca$ for any $n$.

In the next pages it will be convenient to consider the following isomorphic inductive family: $\ca_n$ consists of continuous $B^n\bz^p$-periodic functions on $\br^p$, and the embedding is the inclusion. In this way $\ca_\infty$ may be identified with a generalized solenoid C$^*$-algebra (cf. \cite{McCord}, \cite{LP2}). 

Since $\bt_n=\br^p/B^n\bz^p$ is a covering space of $\bt_0:=\bt^p$, the formula of the Dirac operator on $\bt_n$ doesn't change. Therefore, we will consider the following spectral triple
\begin{eqnarray*}
	(C^1(\bt_n),  \bc^{2^{[p/2]}} \otimes L^2(\bt_n,\frac1{{\molt}^n}dm), D=-i\sum_{a=1}^p \eps^a \otimes\partial^a).
\end{eqnarray*}
The aim of this section is to describe the spectral triple on $\bt_n$ in terms of the spectral triple on $\bt_0$. Consider the short exact sequences of groups
\begin{align}
&0\longrightarrow B^{n}\bz^p\longrightarrow B^{n-1}\bz^p\longrightarrow{\bz_B}\longrightarrow 0,\label{seq1}\\
&0\longrightarrow A^{n-1}\bz^p\longrightarrow A^{n}\bz^p\longrightarrow\widehat{\bz_B}\longrightarrow 0\label{seq2},
\end{align}
where ${\bz_B}$ is now identified with the finite group in \eqref{seq1}, hence is a subgroup of $\bt_{n}$.
The central extension $A^{n}\bz^p$ of $\widehat{\bz_B}$ via $A^{n-1}\bz^p$ can be described either with a section $s_{n}:\widehat{\bz_B}\to A^{n}\bz^p$ or via a $A^{n-1}\bz^p$-valued 2-cocycle $\o_{n}(k,k')=s_{n}(k)+s_{n}(k')-s_{n}(k+k')$, see e.g. \cite{Brown}.
 We choose the unique section such that, for any $k\in \widehat{\bz_B}$, $s_{n}(k)\in
A^{n-1} [0,1)^p$, and observe that this is the same as choosing $s_n(k)=A^{n-1}s_1(k)$.
In the same way, the second extension $B^{n-1}\bz^p$ of $\bz_B$ via $B^{n}\bz^p$ can be described either with a section $\widehat{s}_{n}:\bz_B\to B^{n-1}\bz^p$ or via a $B^n\bz^p$-valued 2-cocycle $\widehat{\o}_{n}(k,k')=\widehat{s}_{n}(k)+\widehat{s}_{n}(k')-\widehat{s}_{n}(k+k')$. We choose the unique section such that, for any $k\in \bz_B$, $\widehat{s}_{n}(k)\in
B^{n} [0,1)^p$. 
The following result holds

\begin{proposition} \label{prop-2-4}
Any function $\xi$ on $\bt_i$ can be decomposed as $\xi=\sum_{k\in \widehat{\bz_B}}\xi_k$, where
\[
\xi_k(t)\equiv E_{k}(\xi)(t) = \frac1\molt \sum_{g\in \bz_B}\langle -k,g \rangle \xi(t-g), \quad t\in \bt_i=\br^p/B^i\bz^p.
\]
Moreover, this correspondence gives rise to unitary operators 
\begin{align*}
v_i: L^2(\bt_i, dm/r^i) & \to \sum\nolimits^\oplus_{k\in \widehat{\bz_B}} \ L^2(\bt_{i-1},dm/r^{i-1}) = L^2(\bt_{i-1},dm/r^{i-1})\otimes \bc^\molt \\
 \xi \quad & \mapsto \qquad \sum\nolimits^\oplus_{k\in \widehat{\bz_B}} \ \sigma(k)^{-1}\xi_k \, .
\end{align*}
The multiplication operator by the element $f\in\ca_i$  is mapped to the matrix $M_\molt(f)$ acting on $L^2(\bt_{i-1}, dm/r^{i-1})\otimes \bc^\molt$ given by
\[
M_\molt(f)_{j,k}(t) = \langle s(j)-s(k),t\rangle f_{j-k}(t), \qquad j,k\in \widehat{\bz_B}.
\]
In particular, when $f$ is $B^{i-1}\bz$-periodic, namely it is  a function on $\bt_{i-1}$, then $M_\molt(f)_{j,k}(t)
=f(t)\d_{j,k}$, i.e. a function $f$ on $\bt_{i-1}$ embeds into $\cb(L^2(\bt_{i-1}, dm/r^{i-1}))\otimes M_\molt(\bc)$ as $f\otimes I$.
\end{proposition}
\begin{proof}
The statement follows from the analysis of Proposition \ref{Representation-covering}, in particular here  $b_k=\xi_k$, $M(b) = M_r(f)$.
\end{proof}

\begin{theorem} \label{thm:tripleOnTn}
The Dirac operator $D_n$ acting on $\bc^{2^{[p/2]}} \otimes L^2(\bt_{n},\frac1{{\molt}^{n}}dm)$  gives rise to an operator, which we denote by $\widehat{D}_{n}$, when the Hilbert space is identified with the Hilbert space $\bc^{2^{[p/2]}} \otimes L^2(\bt_0,dm) \otimes (\bc^{\molt})^{\otimes n}$ as above. 
The Dirac operator $\widehat{D}_{n}$ 
has the following form:
\[
\widehat{D}_{n} = V_{n}D_nV_{n}^* = D_0\otimes I- 2\pi \sum_{a=1}^p \eps^a \otimes I \otimes \bigg( \sum_{h=1}^n I^{\otimes h-1} \otimes \diag(s_{h}(\cdot)^a) \otimes I^{\otimes n-h} \bigg),
\]
where  $\diag(s_{h}(\cdot)^a)_{j,k}=\delta_{j,k} s_j(k)^a$ for $j, k\in \widehat{\bz_B}$,  the unitary operator $V_n: \bc^{2^{[p/2]}} \otimes L^2(\bt_n,\frac{1}{r^n}dm) \to \bc^{2^{[p/2]}} \otimes L^2(\bt_0,dm)\otimes (\bc^{r})^{\otimes n}$ is defined as $V_n:= I \otimes [(v_1\otimes \bigotimes_{j=1}^{n-1} I)\circ (v_2\otimes \bigotimes_{j=1}^{n-2} I)\circ \cdots \circ v_n]$.
Moreover, we have the following spectral triple 
\begin{displaymath}
	(\cl_n:=C^1(\bt_n), \bc^{2^{[p/2]}} \otimes L^2(\bt_0, dm) \otimes (\bc^{\molt})^{\otimes n}, \widehat{D}_{n}). 
\end{displaymath}
\end{theorem}
\begin{proof}
First of all we prove the formula for $n=1$. 
We give a formula for $D_1$ acting on $\bc^{2^{[p/2]}} \otimes L^2(\bt_1,\frac1{\molt}dm) \cong \bc^{2^{[p/2]}} \otimes L^2(\bt_0,dm) \otimes \bc^{\molt}$.
Let us denote by $\{ \eta_k\}_{k\in\widehat{\bz_B}}$ a ${\molt}$-tuple of vectors in $\bc^{2^{[p/2]}} \otimes L^2(\bt_0,dm)$, so that
$\xi := \sum_{k\in\widehat{\bz_B}} \s(k)\eta_k$ is an element in $\bc^{2^{[p/2]}} \otimes L^2(\bt_1,\frac1{\molt} dm)$, and $E_k(\xi)=\s(k)\eta_k$, $k\in\widehat{\bz_B}$. Then, for any $t\in\bt_1$, we get
\begin{align*}
\widehat{D}_{1} \big( \sum\nolimits^\oplus_{k\in \widehat{\bz_B}} \eta_k(t) \big)
& = V_{1}D_1V_{1}^* \big( \sum\nolimits^\oplus_{k\in \widehat{\bz_B}} \eta_k(t) \big)  \\
& = \sum\nolimits^\oplus_{j\in \widehat{\bz_B}} \frac1{\molt} \langle s(j),t\rangle  \sum_{g\in\bz_B} \langle -j,g \rangle D \Big( \sum_{ k\in \widehat{\bz_B}} \langle s(k),-t+g\rangle \eta_k(t-g) \Big)  \\
& = \sum\nolimits^\oplus_{j\in \widehat{\bz_B}} \sum_{ k\in \widehat{\bz_B}} \frac1{\molt} \langle s(j),t\rangle  \sum_{g\in\bz_B} \langle k-j,g \rangle D \big(  \langle s(k),-t\rangle \eta_k(t) \big)  \\
& = \sum\nolimits^\oplus_{k\in \widehat{\bz_B}}  \langle s(k),t\rangle   D \big(  \langle s(k),-t\rangle \eta_k(t) \big)  \\
& = -i \sum_{a=1}^p \sum\nolimits^\oplus_{k\in \widehat{\bz_B}}  \langle s(k),t\rangle   \eps^a  \partial^a \big( \langle s(k),-t\rangle \eta_k(t) \big)  \\
& = -i \sum_{a=1}^p \sum\nolimits^\oplus_{k\in \widehat{\bz_B}}    \eps^a \big(  -2\pi is(k)^a\eta_k(t) + \partial^a \eta_k(t) \big)  \\
& =  \sum_{a=1}^p \Big( -2\pi \eps^a \otimes I \otimes \diag( s(k)^a )_{k\in \widehat{\bz_B}} - i \eps^a \otimes \partial^a \otimes I \Big) \sum\nolimits^\oplus_{k\in \widehat{\bz_B}}   \eta_k(t) .
\end{align*}
The formula for $n>1$ can be obtained by iterating the above procedure.
\end{proof}

\subsection{The inductive limit spectral triple}

The aim of this section is to construct a spectral triple for the inductive limit $\varinjlim \ca_n$. We begin with some preliminary results.
A matrix $B\in M_p(\bz)$ is called purely expanding if, for all vectors $v\neq 0$, we have that $\|B^n v\|$ goes to infinity.

\begin{proposition}\label{prop-pur-exp}
Assume $\det B\ne0$, $A=(B^T)^{-1}$. Then the following  are equivalent:
\begin{enumerate}
\item[$(1)$] $B$ is purely expanding, 
\item[$(2)$]  $\|A^n\|\to 0$,
\item[$(3)$] the spectral radius ${\rm spr}(A)<1$,
\item[$(4)$] $\sum_{n\geq 0} \|A^n\|<\infty$.
\end{enumerate}
\end{proposition}
\begin{proof}
$(1)\Leftrightarrow(2)$
Consider a vector $w=B^nv/\|B^nv\|$, then, from the identity
$$
\|B^nw\|=\frac{\|v\|}{\|B^nv\|},
$$
we deduce that $(1)$ is equivalent to $\|B^{-n}v\|\to 0$, for all $v\neq 0$. The latter is equivalent to $(2)$ by the identity $(A^nv,u)=(v,B^{-n}u)$, for any vectors $u,v$.

$(2)\Rightarrow (3)$ We argue by contradiction. Let $\lambda\in {\rm sp}(A)$ have modulus $|\lambda| \geq 1$, and consider an associated eigenvector $v\neq 0$. Then, we have that $\|A^nv\| = |\lambda|^n \|v\| \not\to 0$.

$(3)\Rightarrow (4)$ Let $A= C^{-1}(D+N)C$ be the Jordan decomposition of $A$, where $D$ is the diagonal part, and $N$ the nilpotent one. Then
\begin{equation}\label{stima-norma}
\begin{aligned}
	\|(D+N)^n\|&= \| \sum_{j=0}^{p-1} \binom{n}{j} D^{n-j} N^j\| \leq  \sum_{j=0}^{p-1} \binom{n}{j} \|D^{n-j}\|\\
&= \sum_{j=0}^{p-1} \binom{n}{j} {\rm spr}(A)^{n-j}
\leq {\rm spr}(A)^n \left(\sum_{j=0}^{p-1} n^j {\rm spr}(A)^{-j}\right)\\
&={\rm spr}(A)^n\frac{(n/{\rm spr}(A))^p-1}{n/{\rm spr}(A)-1} < \frac{n^p}{n-1}{\rm spr}(A)^{n-p},
\end{aligned} 
\end{equation}
where we used  $N^p=0$, $\|N^j\|\leq 1$,  $\|D\|={\rm spr}(A)$, so that the series $\sum_{n\geq 0} \|A^n\|$ converges.

$(4)\Rightarrow (2)$ is obvious.
\end{proof}

\begin{theorem}
Assume now that $B$ is purely expanding 
 and consider the C$^*$-algebras $\ca_n=C(\br^p/B^n\bz^p)$, which embed into $M_{2^{[p/2]}}(\bc) \otimes \cb(L^2(\bt_0,dm)) \otimes M_{{\molt}^n}(\bc)$,  and the Dirac operators $\widehat{D}_{n}\widehat{\in} \cb(\ch_0) \otimes UHF({\molt}^\infty)$, where $\ch_0:= \bc^{2^{[p/2]}} \otimes L^2(\bt_0,dm)$.
As a consequence, $\ca_\infty$ embeds in the injective limit 
$$
\varinjlim \cb(\ch_0)\otimes M_{{\molt}^n}(\bc) = \cb(\ch_0) \otimes \mathrm{UHF}({\molt}^\infty)
$$
hence in $\cb(\ch_0)\otimes \car$, where $\car$ is the injective type II$_1$ factor.
Moreover, the  operator $\widehat{D}_{\infty}$ has the following form:
\[
\widehat{D}_{\infty} = D_0 \otimes I - 2\pi \sum_{a=1}^p \eps^a \otimes I \otimes \bigg(  \sum_{h=1}^\infty I^{\otimes h-1} \otimes \diag(s_{h}(\cdot)^a)  \bigg).
\]
In particular, $\widehat{D}_{\infty}$  is affiliated to $\cb(\ch_0) \otimes \car = \cam$ and has the form $D_{0}\otimes I+C$, with $C=C^*\in \cb(\ch_0) \otimes {\rm UHF}(r^\infty) \subset \cb(\ch_0) \otimes \car = \cam$.
\end{theorem}
\begin{proof}
The formula and the fact that $\widehat{D}_{\infty}$ is affiliated to $\cam$ follow from what has already been proved and the following argument.
We posed $s_{n}(k)\in A^{n-1} [0,1)^p$, therefore 
$$
\max_{k\in\widehat{\bz_B}}\|s_{n}(k)\|\leq\sup_{x\in[0,1)^p}\|A^{n-1}x\|\leq \|A^{n-1}\|\sqrt p.
$$
As a consequence, for any $a\in\{1,\ldots,p\}$,
$$
\| \diag(s_n(k)_a)_{k\in\widehat{\bz_B}}\| = \max_{k\in\widehat{\bz_B}} |s_n(k)_a|\leq \max_{k\in\widehat{\bz_B}} \|s_{n}(k)\| \leq \|A^{n-1} \| \sqrt p.
$$
Recalling that $\widehat{D}_{\infty}=D_{0}\otimes I+C$, with $C=2\pi  \displaystyle \sum_{a=1}^p \eps^a \otimes I \otimes \bigg(  \sum_{h=1}^\infty I^{\otimes h-1} \otimes \diag(s_{h}(k)^a)  \bigg)$, we get, by Proposition \ref{prop-pur-exp} and the estimate above, that $C$ is bounded and belongs to $M_{2^{[p/2]}}(\bc) \otimes \bc\otimes \mathrm{UHF}({\molt}^\infty)$, while $D_{0} \widehat{\in} \cb(\ch_0)$.
\end{proof}

\begin{theorem} \label{InductiveLimitTriple}
Let $\{ (\ca_n,\f_n) : n\in\bno\}$ be an inductive system, with $\ca_n\cong\ca_0$, and $\f_n:\ca_n\hookrightarrow \ca_{n+1}$ is the inclusion, for all $n\in\bn$. Suppose that, for any $n\in\bno$, there exists a spectral triple $(\cl_n,\ch_n,\widehat{D}_n)$ on $\ca_n$, with $\ch_n=\ch_0\otimes(\bc^\molt)^{\otimes n}$, $\widehat{D}_n=D_0\otimes I+C_n$,   $C_n\in \cb(\ch_0) \otimes M_{\molt}(\mathbb{C})^{\otimes n}\subset \cb(\ch_0) \otimes UHF(\molt^\infty)$ is a self-adjoint sequence converging to $C\in \cb(\ch_0) \otimes UHF(\molt^\infty)$, and $\widehat{D}_\infty=D_0\otimes I+ C$.  Let $\som$ be the abscissa of convergence of $\zeta_{D_0}$ and suppose that $\res_{s=\som}(\t(D_{0}^2+1)^{-s/2})$ exists and is finite. Let $\cl_\infty:=\cup_{n=0}^\infty \cl_n$. Then  $(\cl_\infty,\cb(\ch_0)\otimes\car, \ch_0\otimes L^2(\car,\tau), \widehat{D}_\infty)$ is a finitely summable, semifinite, spectral triple, with the same Hausdorff dimension of $(\cl_0,\ch_0,D_0)$.
Moreover, the volume of this noncommutative manifold coincides with the volume of $(\cl_0,\ch_0,D_0)$, namely the Dixmier trace $\t_\o$ of $(\widehat{D}_{\infty}^2+1)^{-\som/2}$ coincides with that of $(D_{0}^2+1)^{-\som/2}$ (hence does not depend on $\o$) and may be written as:
$$
\t_\o((\widehat{D}_{\infty}^2+1)^{-\som/2})=\lim_{t\to\infty}\frac1{\log t}\int_0^t \left(\m_{(D_{0}^2+1)^{-1/2}}(s)\right)^\som \,ds.
$$
\end{theorem}
\begin{proof}

As for the commutator condition, we observe that for each $f\in \cl_n$ we have that $[\widehat{D}_{\infty},f]$ is bounded since $[\widehat{D}_{n},f]$ is bounded.

We now show that $\widehat{D}_{\infty}$ has $\t$-compact resolvent, where $\t$ is the unique f.n.s. trace on $\cb(\ch_0)\otimes\car$. Indeed, on a finite factor, any bounded operator has $\t$-finite rank, hence is $\t$-compact. Therefore, since $D_{0}$ has compact resolvent in $\cb(\ch_0)$, $D_{0}\otimes I$ has $\t$-compact resolvent in $\cb(\ch_0)\otimes \car$. We have $(D_{0}\otimes I+C+i)^{-1}=[I+(D_{0}\otimes I+i)^{-1}C]^{-1}(D_{0}\otimes I+i)^{-1} = (D_{0}\otimes I+i)^{-1}[I+C(D_{0}\otimes I+i)^{-1}]^{-1}$, where $I+C(D_{0}\otimes I+i)^{-1}$ and $I+(D_{0}\otimes I+i)^{-1}C$ have trivial kernel and cokernel. Indeed $Ran (I+(D_{0}\otimes I+i)^{-1}C)^\perp=\ker(I+C(D_{0}\otimes I-i)^{-1})$, and $(I+C(D_{0}\otimes I\pm i)^{-1})x=0$ means $(C+D_{0}\otimes I)y=\mp i y$ with $y=(D_{0} \otimes I \pm i)^{-1}x$, which is impossible since $C+D_{0}\otimes I$ is self-adjoint. Moreover, $ker (I+(D_{0}\otimes I+i)^{-1}C)$ is trivial. In fact, $(I+(D_{0}\otimes I\pm i)^{-1}C)x=0$ implies that $(D_{0}\otimes I+C)x=\mp ix$ which is impossible because $D_{0}\otimes I+C$ is self adjoint. Therefore $I+C(D_{0}\otimes I+i)^{-1}$ has bounded inverse, hence $D_{0}\otimes I+C$ has $\t$-compact resolvent.

Since $D_{0}$ has spectral dimension $\som$, $\res_{s=\som}(\t(D_{0}^2+1)^{-s/2})$ exists and is finite.  Then, applying Proposition \ref{sameRes}, in the appendix, we get $\res_{s=\som}(\t(D_{0}^2+1)^{-s/2})=\res_{s=\som}(\t(D_{\infty}^2+1)^{-s/2})$. The result follows by \cite{CRSS}, Thm 4.11.
\end{proof}

\begin{corollary}
Let $(\cl_n,\ch_n,\widehat{D}_{n})$ be the spectral triple on $\bt_n$ constructed in Theorem \ref{thm:tripleOnTn}, and let us set $\cl_\infty:= \cup_{n=0}^\infty \cl_n$, $\cam_\infty := \cb(\ch_0)\otimes \car$, $\ch_\infty := \ch_0\otimes L^2(\car,\tau)$. Then  $(\cl_\infty, \cam_\infty,\ch_\infty,\widehat{D}_{\infty})$ is a finitely summable, semifinite, spectral triple, with Hausdorff dimension $p$.
Moreover, the Dixmier trace $\t_\o$ of $(\widehat{D}_{\infty}^2+1)^{-p/2}$ coincides with that of $(D_{0}^2+1)^{-p/2}$ (hence does not depend on $\o$) and may be written as:
$$
\t_\o((\widehat{D}_{\infty}^2+1)^{-p/2})=\lim_{t\to\infty}\frac1{\log t}\int_0^t \left(\m_{(D_{0}^2+1)^{-1/2}}(s)\right)^p \,ds.
$$
\end{corollary}
\begin{proof}
By construction, $\cl_\infty$ is a dense $*$-subalgebra of the C$^*$-algebra $\ca_\infty\subset\cam_\infty$.
The thesis follows from Theorem \ref{InductiveLimitTriple} and the above results.
\end{proof}

\section{Self-coverings of rational rotation algebras}
\subsection{Coverings of noncommutative tori}

Let $A_\th$ be the noncommutative torus generated by $U,V$ with $UV=e^{2\pi i\th}VU$, $\th\in[0,1)$. Given a matrix $B\in M_2(\bz)$, $\det B\ne0$, $B=\begin{pmatrix}a&b\\c&d\end{pmatrix}$, we may consider the C$^*$-subalgebra $A_\th^B$ generated by the elements
\begin{equation}\label{Bgen}
U_1=U^aV^b,\qquad
V_1=U^cV^d.
\end{equation}
We may set $W(n):=U^{n_1} V^{n_2}$ with $n\in \bz^2$. By using the commutation relation between $U$ and $V$, it is easy to see that
\begin{equation}\label{potenze-W}
\begin{aligned}
W(m)W(n) & = e^{-2\pi i \theta m_2n_1},\\
W(n)^k & = e^{-\pi i \theta k(k-1)n_1n_2}W(kn), \qquad \forall k\in \bz . 
\end{aligned} 
\end{equation}

\begin{lemma}
\item{$(i)$} $A_\th^B = A_\th \Longleftrightarrow r=|\!\det B|=1$.
\item{$(ii)$} $A_\th^B \cong A_{\th'}$, where $\th'=r\th$.
\item{$(iii)$} $A_\th^B \cong A_\th$ {\it iff} $r\equiv_q\pm 1$.
\end{lemma}
\begin{proof}
$(i) (\Leftarrow)$ By using equation \eqref{potenze-W} it can be shown that  the generators of $(A_\th^B)^{B^{-1}}$ are
\begin{align*}
U_2&=e^{\pi i\th b d(1-a+c)\det B}U,\\
V_2&=e^{\pi i\th ac(1+b-d)\det B}V.
\end{align*} 
Hence $A_\th=(A_\th^B)^{B^{-1}}\subset A_\th^B\subset A_{\th}$, namely these algebras coincide.

$(ii)$ We compute the commutation relations for $U_1$ and $V_1$, getting
$U_1V_1=e^{2\pi i\det B\th}V_1U_1$. Since $A_{\det B\th}\cong A_{r\th}$, the statement follows.

$(iii)$ We have $A_\th\cong A_{\th'}\Leftrightarrow \th\pm \th'\in\bz\Leftrightarrow(r\pm1)\th\in\bz$. This means in particular that $\th=p/q$, for some relatively prime $p,q\in\bn$, and $r\equiv_q\pm1$. 

$(i) (\Rightarrow)$ Finally, we observe that $A_\th=A_\th^B \Rightarrow A_\th\cong A_{\th'}$. In the following section (Remark \ref{covprops}) we show that $A_\th^B$ is a proper subalgebra of $A_{\th}$ when $r\ne\pm1$, thus completing the proof of $(i)$. 
\end{proof}

On the one hand, the previous Lemma shows that, setting $\th_n=r^{-n}\th$, the algebras $A_{\th_n}$ form an inductive family, where $A_{\th_{k-1}}$ can be identified with the subalgebra $A_{\th_k}^{B}$ of $A_{\th_k}$. The inductive limit is a noncommutative solenoid according to \cite{LP2,LP}.

On the other hand, since in this paper we are mainly concerned with self-coverings, we will, in the following, consider only the rational case $\th=p/q$, with $r\equiv_q\pm1$. Possibly replacing $B$ with $-B$, this is the same as assuming $\det B\equiv_q1$.

\subsection{The $C^*$- algebra, a spectral triple and the self-covering}
\subsubsection*{A  description of $A_\theta$}

We are now going to give a description of the rational rotation algebra making small modifications to the description of $A_\theta$, $\theta=p/q\in\mathbb{Q}$, seen in \cite{BEEK}.
Consider the following matrices 
\begin{eqnarray*}
  (U_0)_{hk} = \delta_{h,k}e^{2\pi i(k-1) \theta }, \quad
  (V_0)_{hk} = \delta_{h+1,k} +\delta_{h,q}\delta_{k,1} \in M_q(\mathbb{C})
\end{eqnarray*}
and
\begin{eqnarray*}
   J=\left(
 		\begin{array}{cc}
 		0 & 1 \\ -1 & 0
 		\end{array}
 		\right)\in M_2(\mathbb{C}).
\end{eqnarray*}
We have that $U_0V_0=e^{2\pi i\theta}V_0U_0$.  
Let $n=(n_1,n_2)\in \mathbb{Z}^2$ and set $W_0(n)\stackrel{def}{=}U_0^{n_1}V_0^{n_2}$,  $\widetilde{\g}_n(f)(t):=\ad(W_0(J n))[f(t-n)]=V_0^{n_1}U_0^{-n_2}f(t-n)U_0^{n_2}V_0^{-n_1}$. Since formula \eqref{potenze-W} holds whenever two operators satisfy the commutation relation $UV=e^{2\pi i\theta}VU$, the following formula holds
\begin{equation}\label{potenze-W0}
W_0(n)^k = e^{-\pi i \theta k(k-1)n_1n_2}W_0(kn) \qquad \forall k\in \bz . 
\end{equation}

We have the following description of $A_\theta$  (cf. \cite{BEEK})
$$
A_\theta=\{f\in C(\mathbb{R}^2, M_q(\mathbb{C})) \, : \, f  = \widetilde{\g}_{n}(f),  n\in\bz^2 \}.
$$
This algebra comes with a natural trace 
$$
\tau(f):= \frac{1}{q}\int_{\bt_0} \tr(f(t))dt, 
$$
where we are considering the Haar measure on $\bt_0$ and $\tr(A)=\sum_i a_{ii}$. We observe that the function $\tr(f(t))$ is $\bz^2$-periodic.
The generators of the algebra are 
\begin{align*}
	U(t_1,t_2)&=e^{2\pi i\theta t_1} U_0, \\
	V(t_1,t_2)&=e^{2\pi i\theta t_2} V_0.
 \end{align*}
 They satisfy the following commutation relation
 \begin{displaymath}
	U(t)^\alpha V(t)^\beta =e^{2\pi i\theta \alpha\beta}V(t)^\beta U(t)^\alpha , \quad \a,\b\in\bz.
\end{displaymath}
We set $W(n,t)=U(t)^{n_1}V(t)^{n_2}$, $\forall t\in\br^2$, $n\in\bz^2$,  and  note that
\begin{align*}
 		W(m,t)W(n,t) & = e^{2i\pi\theta (m,Jn)} W(n,t)W(m,t),\\
 		U(t) & = W((1,0),t),\\
 		V(t) & = W((0,1),t).
\end{align*}
We observe that $\widetilde{\g}_n(f)(t)=\ad(W(J n,t))[f(t-n)]$, $\forall t\in\br^2$, $n\in\bz^2$.
 
\subsubsection*{A spectral triple for $A_\theta$}

Define 
\begin{displaymath}
		\cl_\theta :=\left\{\sum_{r,s}a_{rs}U^rV^s :  (a_{rs})\in S(\mathbb{Z}^2) 	\right\},
\end{displaymath}
where $S(\mathbb{Z}^2)$ is  the set of rapidly decreasing sequences. It is clear that the derivations $\partial_1$ and $\partial_2$, defined as follows on the generators, extend to $\cl_\theta$
\begin{eqnarray*}
	\partial_1(U^hV^k)&=&2\pi ihU^hV^k\\
	\partial_2(U^hV^k)&=&2\pi ikU^hV^k.
\end{eqnarray*} 
Moreover, the above derivations extend to densely defined derivations both on $A_\theta$ and $L^2(A_\theta,\tau)$.

We still denote these extensions with the same symbols. 
We may consider the following spectral triple (see \cite{Conn2}, or section 12.3 in \cite{GBFV}):
 \begin{eqnarray*}
	(\cl_\theta ,  \bc^{2} \otimes L^2(A_\theta,\tau), D=-i(\eps^1 \otimes \partial_1+\eps^2 \otimes \partial_2)),
\end{eqnarray*}
where $\eps^1, \eps^2$ denote the Pauli matrices. In order to fix the notation we recall that the Pauli matrices are self-adjoint, in particular they satisfy the condition $(\eps^k)^2=I$, $k=1,2$.

\subsubsection*{The noncommutative self-covering}
Let  $\ca\doteq A_\theta$ be a rational rotation algebra, $\th=p/q$,  $B\in M_2(\mathbb{Z})$ be a matrix such that $\det B\equiv_q1$, $\molt:=|\!\det B|>1$, and set $C_B = \begin{pmatrix} d & -c \\ -b & a\end{pmatrix}$ the cofactor matrix of $B$, and $A=(B^T)^{-1}$. Then a self-covering of $\ca$ may be constructed in analogy with the construction for the classical torus. Consider the $C^*$-algebra
$$
\cb := \{f\in C(\mathbb{R}^2, M_q(\mathbb{C})) \, : \, f =\widetilde{\g}_{B n}(f),  n\in\bz^2 \}. 
$$
This algebra  is generated by the elements
\begin{equation}\label{UBVB}
\begin{aligned}
	U_\cb(t)&=e^{\pi i\th b d(1-a+c)}e^{2\pi i\theta \langle Ae_1,t\rangle} W_0(C_Be_1), 
	\\
	V_\cb(t)&=e^{\pi i\th ac(1+b-d)}e^{2\pi i\theta \langle Ae_2,t\rangle} W_0(C_Be_2),
 \end{aligned}
 \qquad e_1=\begin{pmatrix}1\\0\end{pmatrix},\quad e_2=\begin{pmatrix}0\\1\end{pmatrix},
 \end{equation}
and can be endowed with a natural trace 
$$
\tau_1(f):= \frac{1}{q |\!\det B |}\int_{\bt_1} \tr(f(t)) dt, \qquad f\in\cb.
$$
The action $\widetilde{\gamma}$ of $\bz^2$ on $\cb$, being trivial when restricted to $B\bz^2$,  induces an action  of $\bz_B$.

\begin{remark}\label{covprops}
The algebra $\ca$ coincides on the one hand  with the fixed point algebra w.r.t. the action of $\bz_B$, and on the other hand with the algebra $\cb^B$ constructed as in \eqref{Bgen}. In fact, by using \eqref{potenze-W0}, a straightforward computation shows that  the elements $U,V$ that generate $\ca$ are given by $U=U_\cb^aV_\cb^b$, $
V=U_\cb^cV_\cb^d$, proving that the inclusion $\ca\subset\cb$ is a non-abelian self-covering w.r.t. the group $\bz_B$.
Since $C(\bt_1)$ is the center of $\cb$, the action of $\bz_B$ restricts to the action of $\bz_B$ on  $C(\bt_1)$  described in the previous section. Therefore,   the covering we are studying is  regular according to Definition  \ref{unitaries}, with the same map $\s$  as that for the commutative torus, see \eqref{map-sigma}. In particular the action of $\bz_B$ is faithful (cf. Remark \ref{regcov}), hence the inclusion $\ca\subset\cb$ is strict since $|\bz_B|=|\det B|>1$.
\end{remark}
\begin{proposition}
The GNS representation  $\pi_1: \cb\to B(L^2(\cb,\tau_1))$ is unitarily equivalent to the representation obtained by $\pi_0: \ca\to B(L^2(\ca,\tau))$ according to Proposition \ref{Representation-covering}.
\end{proposition}
\begin{proof}
It is enough to prove that $\tau_1=\tau_0\circ E$, where $E$ is the conditional expectation from $\cb$ to $\ca$. We have that
\begin{eqnarray*}
\tau_0[E(f)]&=& \frac{1}{q}\int_{\bt_0} \tr[E(f)(t)]=\\
&=&  \frac{1}{q r}\int_{\bt_0} \sum_{n\in\bz_B}\tr[\gamma_n(f)(t)]=\\
&=&  \frac{1}{q r}\int_{\bt_0} \sum_{n\in\bz_B}\tr[f(t-n)]=\\
&=&  \frac{1}{q r}\int_{\bt_1} \tr[f(t)]=\tau_1(f).
\end{eqnarray*}
\end{proof}

\subsection{Spectral triples on noncommutative covering spaces of $A_\theta$}

Given the integer-valued matrix $B\in M_2(\bz)$ as above, there is an associated endomorphism  $\alpha: A_\theta\to A_\theta$ defined by $\a(f)(t)=f(Bt)$. Then, we consider the inductive limit $\ca_\infty=\displaystyle\varinjlim\ca_n$ described in \eqref{ind-lim-diag}, where $\ca_n=\ca$ for any $n$.

As in Section \ref{Tori-cov}, it will be convenient to consider the following isomorphic inductive family: $\ca_n$ consists of continuous $B^k\bz^2$-invariant  matrix-valued functions on $\br^2$, i.e
$$
\ca_k:=\{f\in C(\mathbb{R}^2, M_q(\mathbb{C})) \, : \, f =\widetilde{\g}_{B^k n}(f),  n\in\bz^2 \},
$$ 
with trace
$$
\tau_k(f)=\frac{1}{q |\!\det B^k|}\int_{\bt_k}\tr(f(t))dt,
$$
and the embedding is unital inclusion $\alpha_{k+1,k}: \ca_k\hookrightarrow \ca_{k+1}$.  In particular, $\ca_0=\ca$, and $\ca_1=\cb$. This means that $\ca_\infty$ may considered as a generalized solenoid C$^*$-algebra  (cf. \cite{McCord}, \cite{LP2}).

On the $n$-th noncommutative covering $\ca_n$, the formula of the Dirac operator doesn't change and we can consider the following spectral triple
\begin{eqnarray*}
	(\cl_\theta^{(n)} ,  \bc^{2} \otimes L^2(\ca_n,\tau), D=-i(\eps^1 \otimes \partial_1 + \eps^2 \otimes \partial_2)).
\end{eqnarray*}
The aim of this section is to describe the spectral triple on $\ca_n$ in terms of the spectral triple on $\ca_0=A_\theta$.

We will consider the two central extensions \eqref{seq1} and \eqref{seq2}  (case $p=2$) with the associated sections $s_{n}:\widehat{\bz_B}\to A^{n}\bz^2$ and $\widehat{s}_{n}:\bz_B\to B^{n-1}\bz^2$ defined earlier. 

The following result holds:

\begin{theorem} \label{teo:1}
Any  $b$ in $\ca_i$ can be decomposed as $b=\sum_{k\in \widehat{\mathbb{Z}_B}} b_k$, where \begin{equation} 
	b_k(t)=\frac{1}{\molt} \sum_{g\in \mathbb{Z}_B} \langle -k,g\rangle \g_{g}(b(t))\in (\ca_i)_k.
\end{equation}	
Let $u_g$ be the unitary operator on $L^2(\ca_i,\tau_i)$ implementing the automorphism $\g_g$. Then, any $\xi\in L^2(\ca_i,\t_i)$ can be decomposed as $\xi=\sum_{k\in \widehat{\mathbb{Z}_B}}\xi_k$, where
\begin{equation} 
	\xi_k(t)=\frac{1}{\molt} \sum_{g\in \mathbb{Z}_B} \langle -k,g\rangle u_{g}(\xi(t)).
\end{equation}	
Moreover, this correspondence gives rise to unitary operators $v_{i}:L^2(\ca_{i},\tau_i) \to L^2(\ca_{i-1},\tau_{i-1})\otimes \mathbb{C}^\molt$ defined by $v_{i}(\xi)= \{\sigma(k)^{-1}\xi_k\}_{k\in\widehat{\bz_B}}$. 
The multiplication operator by an element $f$ on $\ca_{i}$ is mapped to the matrix $M_\molt(f)$ acting on $L^2(\ca_{i-1}, \tau_{i-1})\otimes \mathbb{C}^\molt$ given by
\[
M_\molt(f)_{h,k}(t)=\langle s(k)-s(h),-t\rangle f_{h-k}(t), \qquad t\in\br^2, h,k\in\widehat{\bz_B}.
\]
\end{theorem}
\begin{proof}
The statements follow as in Proposition \ref{prop-2-4}.
\end{proof}

\begin{theorem}
Set $\ch_0:=\bc^{2}\otimes L^2(\ca_0,\tau_0)$. Then the Dirac operator $D_n$ acting on $\bc^{2} \otimes L^2(\ca_n,\tau_n)$ gives rise to the operator  $\widehat{D}_{n}$  when the Hilbert space is identified with $\ch_0\otimes(\bc^\molt)^{\otimes n}$ as above. 
Moreover, the Dirac operator $\widehat{D}_{n}$ has the following form:
\[
\widehat{D}_{n}:=V_nD_nV_n^*=D_0\otimes I-2\pi \sum_{a=1}^2 \eps^a\otimes  I\otimes\bigg( \sum_{j=1}^n I^{\otimes j-1} \otimes \diag(s_{j}(k)^a)_{ k\in \widehat{\bz_B} } \otimes I^{\otimes n-j} \bigg),
\]
where  $V_n: \bc^{2} \otimes L^2(\ca_n,\t_n) \to \ch_0 \otimes (\bc^{r})^{\otimes n}$ is defined as $V_n:= I \otimes [(v_1\otimes \bigotimes_{j=1}^{n-1} I)\circ (v_2\otimes \bigotimes_{j=1}^{n-2} I)\circ \cdots \circ v_n]$.
\end{theorem}
\begin{proof}  
We prove the formula for $n=1$, the case $n>1$ can be obtained by iterating the procedure.
Let us denote by $\{ \eta_k \}_{k\in\widehat{\bz_B}}$ an element in $\bc^{2} \otimes L^2(\ca_0,\tau_0)$.
\begin{align*}
V_1D_1V_1^* \big( & \sum\nolimits^\oplus_{k\in\widehat{\bz_B}} \eta_k(t) \big)
 = V_1D_1 \big( \sum_{k\in\widehat{\bz_B}}  \langle s(k),-t\rangle \eta_k(t) \big) \\
& = \sum\nolimits^\oplus_{j\in\widehat{\bz_B}}\langle s(j),t\rangle \frac1{\molt} \sum_{g\in\bz_B} \langle -j,g\rangle u_g \Big( \sum_{k\in\widehat{\bz_B}} D \big( \langle s(k),-t\rangle \eta_k(t) \big) \Big) \\
& \stackrel{(a)}{=} \sum\nolimits^\oplus_{j\in\widehat{\bz_B}} \sum_{k\in\widehat{\bz_B}} \langle s(j),t\rangle \frac1{\molt} \sum_{g\in\bz_B} \langle -j,g\rangle    D \big( \langle s(k),-t+g\rangle \eta_k(t) \big)  \\
& = \sum\nolimits^\oplus_{j\in\widehat{\bz_B}} \sum_{k\in\widehat{\bz_B}} \langle s(j),t\rangle \frac1{\molt} \sum_{g\in\bz_B} \langle k-j,g\rangle    D \big( \langle s(k),-t\rangle \eta_k(t) \big)  \\
& = \sum\nolimits^\oplus_{k\in\widehat{\bz_B}} \langle s(k),t\rangle  D \big( \langle s(k),-t\rangle \eta_k(t) \big)  \\
& = -i \sum\nolimits^\oplus_{k\in\widehat{\bz_B}} \langle s(k),t\rangle   \sum_{a=1}^2 \langle s(k),-t\rangle \eps^a \big(  -2\pi i s(k)^a \eta_k(t) + \partial^a \eta_k(t) \big)  \\
& = \left( -i \sum_{a=1}^2 \eps^a \otimes \partial^a \otimes I  -2\pi  \sum_{a=1}^2 \eps^a \otimes I \otimes \diag ( s(k)^a )_{k\in\widehat{\bz_B}}  \right)  \sum\nolimits^\oplus_{k\in\widehat{\bz_B}} \eta_k(t) \,,
\end{align*}
where in $(a)$ we used the facts that $u_g\circ D = D\circ u_g$, and $u_g\equiv id$ on $\bc^2\otimes L^2(\ca_0,\t_0)$.
\end{proof}

\subsection{The inductive limit spectral triple}

\begin{proposition}
The C$^*$-algebra $\ca_n$ embeds into $\cb(\ch_0) \otimes \cam_{\molt^n}(\bc)$. As a consequence, $\ca_\infty$ embeds into the injective limit 
$$
\varinjlim \cb(\ch_0) \otimes \cam_{\molt^n}(\bc) = \cb(\ch_0) \otimes \mathrm{UHF}(\molt^\infty)
$$
hence in $\cb(\ch_0) \otimes \car$, where $\car$ is the injective type II$_1$ factor.
\end{proposition}

\begin{theorem} 
Assume that $B$ is purely expanding and that $\det(B)\equiv_q 1$. Let us set $\cl_\theta =  \cup_n\cl_\theta ^{(n)}$, $\cam = \cb(\ch_0) \otimes \car$, and define 
\[
\widehat{D}_{\infty} := D_0 \otimes I -2\pi \sum_{a=1}^2 \eps^a \otimes I \otimes \bigg(  \sum_{j=1}^\infty I^{\otimes j-1} \otimes \diag(s_{j}(k)^a)_{ k\in \widehat{\bz_B} }  \bigg).
\] 
Then $(\cl, \cam, \ch_0 \otimes L^2(\car,\tau),\widehat{D}_{\infty})$ is a finitely summable, semifinite, spectral triple, with Hausdorff dimension $2$.
Moreover, the Dixmier trace $\t_\o$ of $(\widehat{D}_{\infty}^2+1)^{-1}$ coincides with that of $(D_{0}^2+1)^{-1}$ (hence does not depend on the generalized limit $\o$) and may be written as:
$$
\t_\o((\widehat{D}_{\infty}^2+1)^{-1})=\lim_{t\to\infty}\frac1{\log t}\int_0^t\left(\m_{(D_{0}^2+1)^{-1/2}}(s)\right)^2\,ds.
$$
\end{theorem}
\begin{proof}
The formula for $\widehat{D}_{\infty}$ follows from what has already been proved. We want to prove that $\widehat{D}_{\infty}$ is of the form $D_{0}\otimes I+C$, with $C = -2\pi  \displaystyle \sum_{a=1}^2 \eps^a \otimes I \otimes \bigg( \sum_{j=1}^\infty I^{\otimes j-1} \otimes \diag(s_{j}(k)^a)  \bigg) \in \cb(\ch_0) \otimes \car$ and $\widehat{D}_{\infty}\widehat{\in} \cb(\ch_0)\otimes\car$.

By construction, $\cl_\theta$ is a dense $*$-subalgebra of the C$^*$-algebra $\ca_\infty\subset\cam$. We now prove that $\widehat{D}_{\infty}$ is affiliated to $\cam$. We posed $s_{n}(k)\in A^{n-1} [0,1)^2$, therefore
$$
\max_{k\in\hat{\bz_B}}\|s_{n}(k)\|\leq\sup_{x\in[0,1)^2}\|A^{n-1}x\|\leq \|A^{n-1}\|\sqrt 2.
$$
As a consequence, for $a=1,2$, $j\in\bn$,
$$
\| \diag(s_j(k)^a)\|=\max_{k\in\hat{\bz_B}}|s_j(k)^a|\leq\max_{k\in\hat{\bz_B}}\|s_{j}(k)\|\leq \|A^{j-1}\|\sqrt 2.
$$
By  Proposition \ref{prop-pur-exp} and the estimate above, we get that $C$ is bounded and belongs to $M_{2}(\bc) \otimes \bc \otimes \mathrm{UHF}(\molt^\infty)$, while $D_{0}\otimes I \widehat{\in} \cb(\ch_0) \otimes \bc$.

The thesis follows from Theorem \ref{InductiveLimitTriple} and what we have seen above.
\end{proof}

\section{Self-coverings of crossed products}
\subsection{The $C^*$-algebra, its spectral triple and the self-covering}
\subsubsection*{The algebra and the noncommutative covering}

Let $B\in M_p(\bz)$, with $r=|\!\det(B)|>1$, and set $A=(B^T)^{-1}$. Consider  a finitely summable spectral triple $(\cl_\cz,\ch,D)$ on the C$^*$-algebra $\cz$ and assume the following:
\begin{itemize}
\item there is an action $\rho:G_1=A\mathbb{Z}^p\to \text{Aut}(\cz)$;
\item $\displaystyle\sup_{g\in G_1} \| [D,\rho_g(a)] \| <\infty$, for any $a\in \cl_\cz$.
\end{itemize}

Assuming, for simplicity, that $\cz \subset\cb(\ch)$, recall that the crossed product $\ca_{G_1}= \cz\rtimes_{\rho}{G_1}$ is the C$^*$-subalgebra of $\cb(\ch\otimes \ell^2(G_1))$ generated by $\pi_{G_1}(\cz)$ and $U_h$, $h\in G_1$, where 
\begin{align*}
(\pi_{G_1}(z)\xi)(g) & :=\r_g^{-1}(z)\xi(g), \\
(U_h \xi)(g) & :=\xi(g-h), \qquad z\in \cz, g,h\in G_1, \xi\in \ell^2(G_1;\ch)\cong \ch\otimes \ell^2(G_1).
\end{align*}

Set $G_0=\bz^p\subset G_1$. The embedding $\cz\rtimes_\rho G_0\subset \cz\rtimes_\rho G_1$ is a finite covering with respect to the action $\gamma:\bz_B\to {\rm Aut}(\cz\rtimes_\rho G_1)$ defined as
$$
\gamma_j(\sum_{g\in G_1} a_gU_g)=\sum_{g\in G_1} \langle \widehat{s}(j),g\rangle a_gU_g, \qquad j\in\bz_B,
$$
where $\widehat{s}:\bz_B\to \bz^p$ is a section of the short exact sequence 
$$
0\to B\bz^p\to\bz^p\to\bz_B\to 0.
$$ 
In fact, the fixed point algebra of this action is $\ca_{G_0} := \cz\rtimes_\rho G_0$.
\subsubsection*{The spectral triples}

Define the map $\ell: \bz^p\to M_{2^{\ceil{p/2}}}(\bc)$  as $\ell(m):= \sum_{\m=1}^p m_\m \eps_{\m+1}^{(p+1)}$, where $\{\eps_i^{(p+1)}\}_{i=1}^{p+1}$ denote the generators of the Clifford algebra $\bc l(\br^{p+1})$, and $m\in \bz^p$.
\begin{theorem}
The following triple is a spectral triple for the crossed product $\ca_{G_0}= \cz\rtimes_{\rho}{G_0}$
\begin{displaymath}
	(\cl_0 = C_c(\bz^p,\cz), \ch_0 =   \ch \otimes \mathbb{C}^{2^{\ceil{p/2}}} \otimes \ell^2(\bz^p), D_{0} =  D \otimes \eps^{(p+1)}_1 \otimes I  + I \otimes M_\ell ).
\end{displaymath} 
where $C_c(\bz^p,\cz) := \{ \sum_{g\in\ \bz^p} \pi_{G_1}(z_g) U_g : z_g\in\cl_\cz, z_g\neq0$ for finitely many $g\in \bz^p\}$, and $M_\ell$ is the operator of multiplication  by the generalized length function $\ell$ (cf. \cite[p. 333]{GBFV}). If the Hausdorff dimension $d(\cl_\cz,\ch,D)=d$, then $d(\cl_0,\ch_0,D_0)=d+p$. 
\end{theorem}
\begin{proof}
The triple in the statement is indeed an iterated spectral triple in the sense of \cite{Skalski}, sec. 2.4. Equivalently, $\ell(g)$ is a proper translation bounded  matrix-valued function (cf. \cite[Remark 2.15]{Skalski} ). For the sake of completeness we sketch the proof of the statement. For the bounded commutator property it is enough to show that the commutators with $\pi_{G_0}(z), z\in\cl_\cz$, and with $U_h, h\in G_0$ are bounded. The norm of the first is bounded by 
$\sup_{g\in G_0} \| [D,\rho_g(a)] \|$, which is finite for any $a\in \cl_\cz$, the norm of the second is bounded by $\|\ell(h)\|$.
We then explicitly compute the eigenvalues of $D_0^2$: they are given by $\l^2+\|g\|_2^2$
, with $\l$ belong to the sequence of eigenvalues of $D$ and $g\in\bz^p$. The compact resolvent property follows.
The formula for the dimension can be obtained as in \cite[Thm.2.7]{Skalski}.
%
%
\end{proof}


In a similar way we define the following spectral triple for the crossed product $\ca_{G_1}= \cz\rtimes_{\rho}{G_1}$

\begin{displaymath}
	(\cl_{1}=C_c(G_1,\cz), \ch_{G_1} =  \ch \otimes \mathbb{C}^{2^{\ceil{p/2}}} \otimes \ell^2(G_1), D_{1} =  D \otimes \eps^{(p+1)}_1 \otimes I  + I \otimes M_{\ell_1} ).
\end{displaymath} 
where  $\ell_1: G_1=A\bz^p\to M_{2^{\ceil{p/2}}}(\bc)$  is defined as $\ell_1(g):= \sum_{\m=1}^p g_\m \eps_{\m+1}^{(p+1)}$, $g\in G_1$.
\begin{remark}
In this case the triple is not an iterated spectral triple in general, but $\ell_1(g)$ is still a proper translation bounded matrix-valued function. An explicit proof may be given as above.
\end{remark}
\subsection*{Regularity and self-covering property}

In order to show that the covering is regular according to Definition \ref{unitaries}, we need to define a map $\s$ which takes values in the spectral subspaces of $\g$. Consider the section $s:\widehat{\bz_B}\to A\bz^p$ defined for the short exact sequence \eqref{seq0}.
Define $\s: \widehat{\bz_B}\to \cu (\cz\rtimes_\rho A\bz^p)$ as 
\begin{eqnarray} 
	\s(k)=U_{s(k)}.
\end{eqnarray}
We observe that $U_{s(k)}\in (\cz\rtimes_\rho A\bz^p)_k$, $k\in\widehat{\bz_B}$.

We first consider the crossed-product C$^*$-algebras $\ca_{G_0}$ and $\ca_{G_1}$ as acting on the Hilbert spaces $\ch \otimes \mathbb{C}^{2^{\ceil{p/2}}} \otimes\ell^2(G_0)$ and $\ch \otimes \mathbb{C}^{2^{\ceil{p/2}}} \otimes\ell^2(G_1)$.  As remarked in section \ref{ST-CS}, a short exact sequence of groups can be described either via a section  $s:\widehat{\bz_B}\to G_1$  or a 2-cocycle $\o:\widehat{\bz_B}\times\widehat{\bz_B}\to G_0$, $\o(j,k)=s(j)+s(k)-s(j+k)$, where $G_1/G_0=\widehat{\bz_B}$. Since $G_1$ is a central extension of $\widehat{\bz_B}$ by $G_0$, the group $G_1$ may be identified with $(G_0,\widehat{\bz_B})$, with $g\in G_1$ identified with $(g-s\circ\zcp(g),\zcp(g))$, $\zcp(g)$ denoting the projection of $g$ to $\widehat{\bz_B}$. The multiplication rule is given by $(a,b)\cdot (a',b')=(a+a'-\o(b,b'),b+b')$, \cite{Brown}.  The above choice of the section $s$ implies that in particular $s(0)=0$, hence $\o(0,g)=\o(g,0)=0$.

Consider the unitary operator 
\begin{align} \label{unitary-cov-crossed}
V  : & \ \xi \in \ell^2(G_1;\ch\otimes \mathbb{C}^{2^{\ceil{p/2}}})   \longrightarrow  V\xi \in \ell^2(G_0\times G_1/G_0;\ch \otimes \mathbb{C}^{2^{\ceil{p/2}}})  \notag \\
& (V\xi)(m,j) := \xi(m+s(j)), \quad m\in G_0, j\in G_1/G_0\,.
\end{align}

\begin{proposition}
The representation  $\pi_{G_1}: \cz \rtimes_\rho G_1\to \ell^2(G_1;\ch\otimes \mathbb{C}^{2^{\ceil{p/2}}})$  is unitarily equivalent, through $V$, to the representation obtained by  $\pi_{G_0}: \cz \rtimes_\rho G_0\to \ell^2(G_0;\ch\otimes \mathbb{C}^{2^{\ceil{p/2}}})$ according to Proposition \ref{Representation-covering}.
\end{proposition}
\begin{proof}
Since $\ca_{G_1}$ is generated by $\pi_{G_1}(z)$, $z\in\cz$, and $U_h$, $h\in G_1$, it is enough to prove the statement for the generators. Observe that, for any $z\in\cz$, $m,n\in G_0$, $j,k\in G_1/G_0$, $\eta\in \ell^2(G_0\times G_1/G_0;\ch\otimes \mathbb{C}^{2^{\ceil{p/2}}})$, we have
\begin{align*}
(V\pi_{G_1}(z)V^*\eta)(n,k) & = (\pi_{G_1}(z)V^*\eta)(n+s(k)) = (\r_{n+s(k)}^{-1}(z)V^*\eta)(n+s(k)) \\
& = \r_{n+s(k)}^{-1}(z)\eta(n,k), \\
(VU_{m+s(j)}V^*\eta)(n,k) & = (U_{m+s(j)}V^*\eta)(n+s(k)) = (V^*\eta)(n-m+s(k)-s(j)) \\
& = \eta(n-m-\o(j,k-j) ,k-j) \,.
\end{align*}
In order to obtain the representation of these operators in $M_{G_1/G_0}(\cb(\ell^2(G_0;\ch\otimes \mathbb{C}^{2^{\ceil{p/2}}}))$, choose any $\f,\psi\in\ell^2(G_0;\ch\otimes \mathbb{C}^{2^{\ceil{p/2}}})$, and denote by $\{e_j\}_{j\in G_1/G_0}$ the canonical basis of $\ell^2(G_1/G_0)$, so that, for any $j,k\in G_1/G_0$, we get
\begin{align*}
\langle \f, (V\pi_{G_1}(z)V^*)_{jk} \psi \rangle & = \langle \f\otimes e_j, V\pi_{G_1}(z)V^*(\psi\otimes e_k) \rangle \\
& = \sum_{i\in G_1/G_0} \sum_{n\in G_0} e_j(i) e_k(i) \langle \f(n), \r_{n+s(i)}^{-1}(z)\xi(n) \rangle \\
& = \d_{jk} \sum_{n\in G_0} \langle \f(n), (\pi_{G_0}( \r_{s(j)}^{-1}(z) )\xi)(n) \rangle,
\end{align*}
which implies that $(V\pi_{G_1}(z)V^*)_{jk} = \d_{jk} \pi_{G_0}( \r_{s(j)}^{-1}(z) )$; analogously, for $m\in G_0$, $\ell\in G_1/G_0$,
\begin{align*}
\langle \f, (VU_{m+s(\ell)}V^*)_{jk} \psi \rangle & = \langle \f\otimes e_j, VU_{m+s(\ell)}V^*(\psi\otimes e_k) \rangle \\
& = \sum_{i\in G_1/G_0} \sum_{n\in G_0} e_j(i)e_k(i-\ell) \langle \f(n), \psi(n-m-\o(\ell,i-\ell) \rangle \\
& = \d_{k,j-\ell} \sum_{n\in G_0}  \langle \f(n), \psi(n-m-\o(\ell,j-\ell) \rangle,
\end{align*}
which implies that $(VU_{m+s(\ell)}V^*)_{jk} = \d_{k,j-\ell} U_{m+\o(\ell,k)}$. On the other hand,
\begin{align*}
M(\pi_{G_1}(z))_{jk} & = U_{s(j)}^*E_{j-k}(\pi_{G_1}(z))U_{s(k)}  = \d_{jk} U_{s(j)}^*\pi_{G_1}(z)U_{s(k)},
\end{align*}
so that
\begin{align*}
\langle \f, M(\pi_{G_1}(z))_{jk} \psi \rangle & = \d_{jk} \langle \f\otimes e_j, U_{s(j)}^*\pi_{G_1}(z)U_{s(k)} (\psi\otimes e_k) \rangle \\
& = \d_{jk} \langle \f\otimes e_j, \pi_{G_1}(\r_{-s(j)}(z))  (\psi\otimes e_k) \rangle \\
& = \d_{jk} \sum_{i\in G_1/G_0} \sum_{n\in G_0} e_j(i)e_k(i) \langle \f(n), \r_n^{-1}(\r_{s(j)}^{-1}(z)) \psi(n) \rangle \\
& = \d_{jk} \sum_{n\in G_0}  \langle \f(n), (\pi_{G_0}(\r_{s(j)}^{-1}(z)) \psi)(n) \rangle ,
\end{align*}
which implies that $M(\pi_{G_1}(z))_{jk} = \d_{jk} \pi_{G_0}( \r_{s(j)}^{-1}(z) )$. Finally,
\begin{align*}
M(U_{m+s(\ell)})_{jk} & = U_{s(j)}^*E_{j-k}(U_{m+s(\ell)})U_{s(k)} = \frac1r \sum_{g\in\bz_B} \langle k-j,g \rangle  U_{s(j)}^* \g_g(U_{m+s(\ell)}) U_{s(k)} \\
& = \frac1r \sum_{g\in\bz_B} \langle k-j,g \rangle \langle \widehat{s}(g),m+s(\ell) \rangle U_{s(j)}^* U_{m+s(\ell)} U_{s(k)} \\
& = \frac1r \sum_{g\in\bz_B} \langle k-j+\ell, g \rangle U_{m+s(\ell)+s(k)-s(j)} = \d_{k,j-\ell} U_{m+\o(\ell,j-\ell)},
\end{align*}
which ends the proof.
\end{proof}

\begin{corollary}
The following diagram commutes:
\begin{equation}
\begin{matrix}
\ca_{G_0}&\longrightarrow&\ca_{G_1}\\
\downarrow&\circlearrowright&\downarrow\\
\cb(\ch\otimes \mathbb{C}^{2^{\ceil{p/2}}} \otimes\ell^2({G_0}))&\longrightarrow&\cb(\ch\otimes \mathbb{C}^{2^{\ceil{p/2}}} \otimes\ell^2({G_0}))\otimes M_\molt(\bc)
\end{matrix}
\end{equation}
where vertical arrows are the representations, the elements of $\ca_{G_1}$ being  identified with matrices as in the previous Proposition, and the horizontal arrows are given by the monomorphisms $a\to M_a$, $(M_a)_{j,k} = \d_{j,k} U^*_{s(j)} a U_{s(j)}$, both for $a\in \ca_{G_0}$ and for $a\in\cb(\ch\otimes\ell^2({G_0}))$.
\end{corollary}

So far we have defined a finite noncommutative  covering. In order to obtain a self-covering, $\cb$ has to be isomorphic to $\ca$, and we have to make further assumptions. Suppose that there exists an automorphism $\b\in\text{Aut}(\cz)$ such that
\begin{align}
&\beta\circ\rho_{Ag}\circ\beta^{-1}=\rho_{g},\qquad g\in \bz^p;\label{outerEq}
\end{align}

\noindent The following result tells us that the above algebras yield a noncommutative self-covering.

\begin{proposition} (\cite{Will})
Under the above hypotheses, the sub-algebra $\ca_{G_0}=\cz\rtimes G_0\subset \ca_{G_1}$  is  isomorphic to $\ca_{G_1}$, the isomorphism being given by
\[
\a:\sum_{g\in\bz^p} a_g U_{Ag}\in \ca_{G_1}\mapsto \sum_{g\in\bz^p} \b(a_g) U_{g}\in \ca_{G_0}.
\]
The map $\a$ may also be seen as an endomorphism of $\ca_{G_1}$.
\end{proposition}

\subsection{Spectral triples on covering spaces of $\cz\rtimes_\rho \bz^p$}

As above, given an integer-valued matrix $B\in M_p(\mathbb{Z})$ we may define an endomorphism $\alpha: \ca_{G_1}\to \ca_{G_1}$. 
Then, we may describe the inductive limit $\ca_\infty=\displaystyle\varinjlim \ca_n$ 
where $\ca_n=\ca_{G_n}$,  $G_n=A^{n}\mathbb{Z}^p$, and the embedding is the inclusion. 
Endow $G_n$ with the length function $\ell_n: G_n\to M_{2^{\ceil{p/2}}}(\bc)$   defined as $\ell_n(g):= \sum_{\m=1}^p g_\m \eps_{\m+1}^{(p+1)}$, $g=(g_1,\ldots,g_p)\in G_n$ ($\ell_n$ is indeed a proper translation bounded matrix-valued function, \cite[Remark 2.15]{Skalski}). Let us observe that $G_n\subset G_{n+1}$ and that  $|G_n/G_{n-1}|=|\!\det B| =: \molt$.

Let us define the action $\rho^{(n)}$ of $G_n$ on $\cz$ as follows:
\[
\rho^{(n)}_{A^n g}=\b^{-n}\circ\rho_{g}\circ\b^n, \quad g\in G_0.
\]

\begin{lemma}
For any $m<n$, $g\in G_m$, we have that $\rho^{(n)}_g=\rho^{(m)}_g$,  namely the family $\{\rho^{(n)}\}_{n\in\bn}$ defines an action $\rho$ of $\cup_n G_n$.
\end{lemma}

\begin{proof}
From equation \eqref{outerEq}, we have
\[
\rho^{(m+1)}_g=\b^{-(m+1)}\circ\rho_{A^{-m-1}g}\circ\b^{m+1}
=\b^{-m}\circ\rho_{A^{-m}g}\circ\b^{m}=\rho^{(m)}_g, \quad g\in G_{m}.
\]
The thesis follows.
\end{proof}

Suppose that
$$
\displaystyle\sup_{g\in G_n} \| [D,\rho_g^{(n)}(a)] \| <\infty,
$$ 
for any $a\in \cl_n := C_c(G_n, \cz)$.
Then, the algebra $\ca_{G_n}$ has a natural spectral triple
\begin{displaymath}
	(\cl_{n}, \ch \otimes \mathbb{C}^{2^{\ceil{p/2}}} \otimes \ell^2(G_n) , D_{n}= D\otimes \eps^{(p+1)}_1 \otimes I +  I \otimes M_{\ell_n}  ).
\end{displaymath}

\begin{remark}
In order to define a spectral triple for $\ca_{G_n}$, we stress that one could make the stronger assumption  
\begin{align*}
 \| [D,\rho_g(a)] \| & \leq c(\rho) \| [D,a] \| , \quad \forall g\in G_1,\\
 \| [D,\beta^k(a)] \| & \leq c(\beta) \| [D,a] \| , \quad \forall g\in G_1, k \in \bz,
\end{align*}
for any $a\in \cl_n$, and some constants $c(\rho)$, $c(\beta)>0$. An even stronger assumption could be 
$$
\displaystyle \| [D,\rho_g^{(n)}(a)] \| = \| [D,a] \|, 
$$ 
for any $a\in \cl_n$ and $g\in G_n$. 
\end{remark}

The aim of this section is to describe the spectral triple on $\ca_{G_n}$ in terms of the spectral triple on $\ca_{G_0}$. Before proceeding, we observe that as in \eqref{unitary-cov-crossed} we may define a family of unitary operators $v_i: \ch \otimes \mathbb{C}^{2^{\ceil{p/2}}} \otimes \ell^2(G_i) \to \ch \otimes \mathbb{C}^{2^{\ceil{p/2}}} \otimes \ell^2(G_{i-1}) \otimes \ell^2(G_i/G_{i-1})$.

\begin{theorem}
Set $\ch_0:= \ch \otimes \mathbb{C}^{2^{\ceil{p/2}}} \otimes \ell^2(G_{0})$. Then the Dirac operator $D_{n}$ acting on $\ch \otimes \mathbb{C}^{2^{\ceil{p/2}}} \otimes \ell^2(G_{n})$  gives rise to the operator $\widehat{D}_n$ when the Hilbert space is identified with  $\ch_0 \otimes\bigotimes_{i=1}^n \ell^2(G_i/G_{i-1})$ as above,  where $G_i/G_{i-1}\cong \widehat{\bz_B}$. 
The Dirac operator $\widehat{D}_{n}$ 
has the following form:
\[
\widehat{D}_{n} := V_nD_{n}V_{n}^* = D_{0} \otimes I^{\otimes n} + C_n,
\]
with $C_n\in \cb(\ch_0) \otimes 	M_\molt(\bc)^{\otimes n}$  defined, for $\eta\in \ell^2(G_0 \times G_1/G_0 \times \ldots \times G_n/G_{n-1} ; \ch \otimes \mathbb{C}^{2^{\ceil{p/2}}})$, as 
$$
(C_n\eta)(m,j_1,\ldots , j_n) := \sum_{h=1}^n (I\otimes \ell_h(s_h(j_h)))(\eta(m,j_1,\ldots , j_n)), 
$$
and $V_n: \ch \otimes \mathbb{C}^{2^{\ceil{p/2}}}  \otimes \ell^2(G_n) \to \ch_0 \otimes \bigotimes_{j=1}^n \ell^2(G_j/G_{j-1})$ given by $V_n:= (v_1\otimes \bigotimes_{j=1}^{n-1} I)\circ (v_2\otimes \bigotimes_{j=1}^{n-2} I)\circ \cdots \circ v_n$.
\end{theorem}
\begin{proof}
For simplicity, we prove the case $n=1$, the case $n>1$ can be proved by iterating the procedure.
For any $\eta \in \ch \otimes \mathbb{C}^{2^{\ceil{p/2}}} \otimes \ell^2(G_{0}) \otimes \ell^2(G_1/G_0) \cong  \ell^2(G_0\times (G_1/G_0); \ch \otimes \mathbb{C}^{2^{\ceil{p/2}}} )$, we get, for $m\in G_0$, $j\in G_1/G_0$,
\begin{align*}
(V_1D_{1} & V_1^*\eta)(m,j) = (D_1V_1^*\eta)(m+s(j)) \\
& = (D \otimes \eps^{(p+1)}_1)(V_1^*\eta)(m+s(j)) +  (I \otimes \ell_1(m+s(j)))(V_1^*\eta)(m+s(j)) \\
& = (D \otimes \eps^{(p+1)}_1)(\eta(m,j)) + (I \otimes \ell_1(m+s(j)))(\eta(m,j)) \\
& = (D \otimes \eps^{(p+1)}_1 + I \otimes \ell_1(m) )(\eta(m,j)) + ( I\otimes \ell_1(s(j)) )(\eta(m,j)) \\
& = (D_{0} \eta) (m,j) + (C_1\eta)(m,j),
\end{align*}
where $(C_1\eta)(m,j) := (I\otimes \ell_1(s(j)))(\eta(m,j))$ belongs to $I \otimes \cb(\mathbb{C}^{2^{\ceil{p/2}}}\otimes \ell^2(G_0\times G_1/G_0))$. We stress that $(C_1\eta)(m,j)$ dos not depend on $m$ because $\ell_1$ is a linear map.
\end{proof}

For any $n\in \bn_0$ and $x\in \cl_n$,  set $L_{D_n}(x):=\|[D_n,x]\|$.
An immediate consequence of the previous result is that, under a suitable assumption, these  seminorms are compatible.

\begin{corollary}
Suppose that
$$
\|[D_0, {\rm Ad}(U_g)(x)]\|=\|[D_0,x]\| \quad \forall x\in \cup_n \cl_n\;, \; \forall g\in \cup_n G_n\; .
$$
Then for any positive integer $m$, we have that
$$
L_{D_{m+1}}(x)=L_{D_{m}}(x) \quad \forall x\in \cl_m\; .
$$

\end{corollary}
\begin{proof}We give the proof for $m=0$.
As in section \ref{section-nc-covering}, the elements in $\ca_1$ may be seen as matrices with entries in $\ca_0$ acting on $\ell^2(G_1/G_0;\ch_0)$. $\ca_0$ itself is then embedded in $\ca_1$ as diagonal matrices, the matrix $M(x)$ associated with $x\in\ca_0$ being $M(x)_{kk}=(\s(k)^* x\s(k))_{kk}=(U_{-s(k)}x U_{s(k)})_{kk}=(\rho_{-s(k)}(x))_{kk}$, where the action $\rho$ has been naturally extended to $\ca_0$. $D_0\otimes I$ may as well be identified with the diagonal matrix $(D_0\otimes I)_{kk}=D_0$, therefore their commutator is the diagonal matrix $([D_0,\rho_{-s(k)}(x)])_{kk}$.
As for the commutator with the second term of $\widehat{D}_{1}$, let us describe the Hilbert space as $\ell^2(G_1/G_0; (\ch  \otimes \ell^2(G_0))\otimes \mathbb{C}^{2^{\ceil{p/2}}})$. Then both $x$ and $C_1$ act as diagonal matrices, whose entries $jj$ are $\rho_{-s(j)}(x)\otimes I$ for the first operator and $I\otimes \ell_1(s(j))$ for the second,
%
showing that the corresponding commutator vanishes.
The thesis now follows by the assumption. 
\end{proof}

\subsection{The inductive limit spectral triple}
The aim of this section is to describe the  Dirac operator on $\ca_\infty$.

\begin{theorem}
Assume $B$ is purely expanding, set $\ch_0:= \ch \otimes \mathbb{C}^{2^{\ceil{p/2}}} \otimes \ell^2(G_{0})$, $\cl= \cup_n\cl_{n}$, $\cam=\cb(\ch_0)\otimes \car$, and define the Dirac operator $\widehat{D}_{\infty}$ as follows:
\[
\widehat{D}_{\infty}:= D_{0}\otimes I_{\textrm{UHF}}+ C,
\]
where $C=\lim C_n$, $C_n=C_n^*\in \cb(\ch_0)\otimes UHF(r^\infty)$.
Then  $(\cl, \cam, \ch_0 \otimes L^2(\car,\tau), \widehat{D}_{\infty})$ is a finitely summable, semifinite, spectral triple, with the same Hausdorff dimension of $(\cl_{0}, \ch_0, D_{0})$ (which we denote by $\som$).
Moreover, the Dixmier trace $\t_\o$ of $(\widehat{D}_{\infty}^2+1)^{-\som/2}$ coincides with that of $(D_{0}^2+1)^{-\som/2}$ (hence does not depend on the generalized limit $\o$) and may be written as:
$$
\t_\o((\widehat{D}_{\infty}^2+1)^{-\som/2})=\lim_{t\to\infty}\frac1{\log t}\int_0^t \left(\m_{(D_{0}^2+1)^{-1/2}}(s)\right)^\som \,ds.
$$
\end{theorem}

\begin{proof}
The Dirac operator $\widehat{D}_{\infty}$ is of the form $D_{0}\otimes I+C$.  First of all, we  prove that $\widehat{D}_{\infty} \widehat{\in} \cb(\ch_0)\otimes\car$ by showing that $C\in \cb(\ch_0)\otimes\car$.
This claim and the formula  follow from what has already been proved and the following argument.
Since we posed $s_{n}(k)\in A^{n-1} [0,1)^p$, by using the properties of the Clifford algebra and the linearity of $\ell_n$, we get
\begin{align*}
\|\ell_n(\sum_{h=1}^n s_{h}(j_h))\| & = \|\sum_{h=1}^n \ell_n(s_{h}(j_h))\| = \|\sum_{h=1}^n s_{h}(j_h)\|\\
& \leq \sum_{h=1}^n \| s_{h}(j_h)) \| \leq  \sqrt{p} \sum_{h=1}^n \| A^{h-1} \|, 
\end{align*}
so that
$$
\| C_n \| = \| \ell_n(\sum_{h=1}^n s_{h}(j_h))\| \leq   \sqrt{p} \sum_{h=1}^n \| A^{h-1} \|.
$$
As  $\widehat{D}_{\infty}=D_{0}\otimes I+C$, we get, by Proposition \ref{prop-pur-exp} and the estimate above, that $C$ is bounded and belongs to $\cb(\ch_0)\otimes UHF(\molt^\infty)$, while $D_{0}\widehat{\in}\cb(\ch_0)$.

Moreover, by construction, $\cl$ is a dense $*$-subalgebra of the C$^*$-algebra $\ca_\infty\subset\cam$.
The thesis follows from Theorem \ref{InductiveLimitTriple} and the above results.
\end{proof}

\begin{remark}
The inclusion $G_n\rightarrow G_{n+1}$ gives rise to inclusions $i_n:\ch_n\rightarrow \ch_{n+1}$ and $j_n:\ca_{G_n}\rightarrow \ca_{G_{n+1}}$ in such a way that\[
\begin{cases}
i_n(a\xi)=j_n(a)i_n(\xi),\\
i_n(D_n\xi)=D_{n+1}i_n(\xi),
\end{cases}
\quad a\in \ca_{G_n},\xi\in\ch_n.
\]
\end{remark}

\section{Self-coverings of UHF-algebras}

\subsection{The $C^*$-algebra, the spectral triple and an endomorphism}

We want to consider the $C^*$-algebra $UHF(\molt^\infty)$. This algebra is defined as the inductive limit of the following  sequence of finite dimensional matrix algebras:   
\begin{eqnarray*}
	M_0 & = & M_\molt(\mathbb{C})\\
	M_n & = & M_{n-1}\otimes M_\molt(\mathbb{C}) \quad n\geq 1,
\end{eqnarray*}
with maps $\phi_{ij}: M_j\to M_i$ given by $\phi_{ij}(a_i)=a_i\otimes 1$.
We denote by $\ca$ the $UHF(\molt^\infty)$ $C^*$-algebra and set $M_{-1}=\mathbb{C}1_\ca$  in the inductive limit defining the above algebra. The $C^*$-algebra $\ca$ has a unique normalized trace that we denote by $\tau$.

Now we follow \cite{Chris}.  Consider the projection $P_n:L^2(\ca,\tau)\to L^2(M_n,{\rm Tr})$, where ${\rm Tr}: M_\molt(\mathbb{C})\to M_\molt(\mathbb{C})$ is the normalized trace, and define 
\begin{eqnarray*}
	Q_n&=& P_n-P_{n-1}, \quad n\geq 0,\\
	E(x)&=&\tau(x)1_\ca\,.
\end{eqnarray*}

\begin{lemma}
The projection $Q_n: L^2(\ca,\tau)\to L^2(M_n,\tau)\ominus L^2(M_{n-1},{\rm Tr})$ ($n\geq 0$) is given by 
\begin{displaymath}
	Q_n(x_0\otimes \cdots\otimes x_n\otimes\cdots ) = x_0\otimes \cdots\otimes x_{n-1}\otimes [x_n-{\rm Tr}(x_n)1_{M_d(\mathbb{C})}] \tau(x_{n+1}\otimes \cdots ),
\end{displaymath}
where ${\rm Tr}: M_\molt(\mathbb{C})\to \mathbb{C}$ is the normalized trace. 
\end{lemma} 
\begin{proof}
The proof follows from direct computations.
\end{proof}

For any $s>1$, Christensen and Ivan (\cite{Chris}) defined the following spectral triple for the algebra $UHF(\molt^\infty)\stackrel{def}{=} \ca$ 
\begin{eqnarray*} 
	(\cl, L^2(\ca,\tau),D_0=\sum_{n\geq 0} \molt^{ns}Q_n )
\end{eqnarray*}
where $\cl$ is the algebra consisting of the elements of $\ca$ with bounded commutator with $D_0$. It was proved that for any such value of the parameter $s$, this spectral triple induces a metric which defines a topology equivalent to the weak$^*$-topology on the state space (\cite[Theorem 3.1]{Chris}).

Introduce the endomorphism of $\ca$ given by the right shift, $\a(x)=1\otimes x$. Then, according to \cite{Cuntz}, we may consider the inductive limit $\ca_\infty=\varinjlim \ca_n$ with $\ca_n=\ca$ as described in \eqref{ind-lim-diag}. 
As in the previous sections, we have the following isomorphic inductive family: $\ca_i$  is defined as
\begin{eqnarray*}
		\ca_0&=& \ca,\\
		\ca_n &=& M_\molt(\mathbb{C})^{\otimes n} \otimes \ca_0,\\
		\ca_\infty &=& \varinjlim \ca_i
\end{eqnarray*}
 and the embedding is the inclusion.

We want to stress that this case cannot be described within the framework considered in  section \ref{NC-coverings}. In fact, it would be necessary to exhibit a finite abelian group that acts trivially on $1_{M_r(\mathbb{C})}\otimes \bigotimes_{i=1}^\infty M_r(\mathbb{C})$ and that has no fixed elements in $M_r(\mathbb{C})\otimes 1_{\bigotimes_{i=1}^\infty M_r(\mathbb{C})}$. However, since all the automorphisms of $M_r(\mathbb{C})$ are inner, there cannot be any such group.

\subsection{Spectral triples on covering spaces of UHF-algebras}

Each algebra $\ca_p$  has a natural Dirac operator (the one considered earlier) 
\begin{eqnarray*} 
	(\cl^p, H=L^2(\ca_p,\tau),D_p=\sum_{n\geq -p} \molt^{ns}Q_n ),
\end{eqnarray*} 
where $\cl_p$ is the algebra formed of the elements of $\ca_p$ with bounded commutator.

\subsubsection*{The spectral triple on  $\ca_1$}

We are going to describe the Dirac operator on the first covering.

\begin{lemma}
Let $\xi_1\otimes \xi_\infty\in L^2(\ca,\tau)$, $\xi_1\in M_r(\bc)$, we have that
\begin{eqnarray*}
	Q_n(\xi_1\otimes \xi_\infty)=\left\{
 	\begin{array}{cc}
 		(1\otimes Q_{n-1})(\xi_1\otimes \xi_\infty) & \textrm{ if $n>0$} \\ 
 		(F\otimes Q_{n-1})(\xi_1\otimes \xi_\infty) & \textrm{ if $n=0$},
 	\end{array}
 	\right.
\end{eqnarray*}
		where $F: M_\molt(\mathbb{C})\to M_\molt(\mathbb{C})^\circ$ is defined by $F(x)=x-tr(x)$, and $M_\molt(\mathbb{C})^\circ$ are the  matrices with trace $0$.
\end{lemma}

\begin{proposition}
The following relation holds:	
\begin{eqnarray*}
		D_1&=& \molt^{-s}F\otimes E+I\otimes D_0.\\
	\end{eqnarray*}
\end{proposition}

\begin{proof}
Let $e_{ij}\otimes x\in \cd(D_1)\subset L^2(\ca_1,\tau)$. We have that
\begin{eqnarray*}
	D_1(e_{ij}\otimes x)&=& \sum_{n\geq -1} \molt^{ns}Q_n(e_{ij}\otimes x)=\\
	&=& \molt^{-s}Fe_{ij}\otimes Ex+\sum_{n\geq 0} \molt^{ns}e_{ij}\otimes (Q_{n} x)=\\
	&=& [\molt^{-s}F\otimes E+I\otimes D_0](e_{ij}\otimes x).
\end{eqnarray*}
The thesis follows by linearity.
\end{proof}

\subsubsection*{The spectral triple on $\ca_n$ and the inductive limit spectral triple}

In this section we will consider the Dirac operators on $\ca_n$ and $\ca_\infty$.

\begin{theorem}
	The Dirac operator $D_{n}$ has the following form 
	\begin{equation}
		D_{n}=I^{\otimes n}\otimes D_0 + \sum_{k=1}^n \molt^{-sk} I_\molt^{n-k}\otimes F\otimes E.
	\end{equation}
\end{theorem}
\begin{proof}

Let $x\in \cd(D_{n})\subset L^2(\ca_{n},\tau)$. We have that
\begin{align*}
	D_{n}x & =  \sum_{k\geq 0} \molt^{(k-n)s}Q_k x =  \sum_{k\geq 0} \molt^{(k-n)s} (I^{\otimes k}\otimes F \otimes E) x \\
	&=  \sum_{k=0}^{n-1} \molt^{(k-n)s} (I^{\otimes k}\otimes F \otimes E) x + \sum_{k\geq n} \molt^{(k-n)s} (I^{\otimes k}\otimes F \otimes E) x \\
	&=  \sum_{h=1}^{n} \molt^{-sh} (I^{\otimes (n-h)}\otimes F \otimes E) x + (I^{\otimes n} \otimes D_0) x.
\end{align*}
\end{proof}

\begin{corollary}
		The Dirac operator $D_{\infty}$ has the following form
	\begin{equation}
		D_{\infty} = I_{-\infty,-1} \otimes D_0 +\sum_{k=1}^\infty \molt^{-sk} I_{-\infty,-k-1}\otimes F\otimes E,
	\end{equation}
	where $I_{-\infty,k}$ is the identity on the factors with indices in $[-\infty,k]$.
\end{corollary}

\begin{theorem} 
Set $\cl= \cup_n\cl_n$, $\cam = \car \otimes \cb(L^2(\ca_{0},\tau))$. Then the triple $(\cl,\cam,L^2(\car, \tau)\otimes L^2(\ca_0, \tau),D_{\infty})$ is a finitely summable, semifinite, spectral triple, with Hausdorff dimension $\frac{2}{s}$.
Moreover, the Dixmier trace $\t_\o$ of $(D_{\infty}^2+1)^{-1/s}$ coincides with that of $(D_{0}^2+1)^{-1/s}$ (hence does not depend on $\o$) and may be written as:
$$
\t_\o((D_{\infty}^2+1)^{-1/s})=\lim_{t\to\infty}\frac1{\log t}\int_0^t\left(\m_{(D_{0}^2+1)^{-1/2}}(s)\right)^{\frac{2}{s}}\,ds.
$$
\end{theorem}
\begin{proof}
By construction, $\cl$ is a dense $*$-subalgebra of the C$^*$-algebra $\ca_\infty\subset\cam$. Since $D_{\infty} = I_{-\infty,-1} \otimes D_0 + C$, where $C\in\car\otimes I$, $D_\infty$ is affiliated to $\cam$. 

The thesis follows from Theorem \ref{InductiveLimitTriple} and what we have seen above.
\end{proof}

\section{Inductive limits and the weak$^*$-topology of their state spaces}

First of all, we recall some definitions. 
Let $(\cl , H, D)$ be a spectral triple over a unital $C^*$-algebra $\ca$. Then we can define the following pseudometric on the state space
\begin{displaymath}
	\rho_D(\phi,\psi) = \sup\{|\phi(x)-\psi(x)| : x\in \ca, L_D(x)\leq 1\}, \quad \phi, \psi \in S(\ca),
\end{displaymath}
where $L_D(x)$ is the seminorm $\Arrowvert [D,x]\Arrowvert$.

We have the following result proved by Rieffel.
\begin{theorem} (\cite{Rief})
	The pseudo-metric $\rho_D$ induces a topology equivalent to the weak$^*$-topology if and only if the ball 
	\begin{displaymath}
		B_{L_D}:=\{x\in \ca : L_D(x) \leq 1\}.
	\end{displaymath}
is totally bounded in the quotient space $\ca/\mathbb{C}1$
\end{theorem}
If the above condition is satisfied, the seminorm $L_D$ is said a \emph{Lip-norm on $\ca$}. In our examples we determined a semifinite spectral triple on $\ca_\infty$.
Our aim is to prove that the  seminorm $L_{\widehat{D}_\infty}$, restricted to $\ca_n$, is a Lip-norm equivalent to $L_{D_n}$, for any $n$, while it is not a Lip-norm on the whole inductive limit  $A_\infty$. 
Therefore, the pair $(\ca_\infty,L_{\widehat{D}_\infty})$ is not a quantum compact metric space, whilst $\ca_\infty$  is topologically compact (i.e. it is a unital C$^*$-algebra).

\begin{theorem}
Consider the Dirac operators $\widehat{D}_\infty$ determined in the previous sections. Then the sequence of  the normic radii of the balls $B_{L_{D_n}}$  diverges. In particular, the seminorm $L_{\widehat{D}_\infty}$ on the inductive limit  is not Lipschitz.
\end{theorem} 
\begin{proof}
Our aim is to show that $B_{L_{\widehat{D}_\infty}}$ is  unbounded. Actually, we will exhibit a sequence  in $B_{L_{\widehat{D}_\infty}}$ with constant seminorm and diverging quotient norm, which means that it is   an unbounded set in $\varinjlim \ca_k/\bc$.

 In the first place we consider the cases of the commutative and noncommutative torus.
The noncommutative rational torus has centre isomorphic to the algebra of continuous functions on the torus. Thus, it is enough to exhibit a sequence only in the case of the torus.

Consider the following sequence
\begin{displaymath}
	x_k = e^{2\pi i (A^{k}e_1, t)}
\end{displaymath}
where $A:=(B^T)^{-1}$.

Each $x_k\in C(\bt_k)\subset \lim_i C(\bt_i)$.
We have that
\begin{align*}
	 [D_k,x_k]  & =  \sum_a \eps^a [\partial_a ,x_k] \leq \sum_a [\partial_a ,x_k]  \\
	&=  \sum_a  \partial_a (x_k) = \sum_a   2\pi i ( A^{k}e_1,e_a ) x_k  \\
	&\leq     2p\pi  \| A^{k}e_1 \| \leq    2p\pi  \| A^k\|  \to 0
\end{align*}
where we used  Proposition \ref{prop-pur-exp}.

Consider the sequence $y_k:=x_k/ \|[D_k,x_k]\|$. This sequence has constant seminorm $L_{\widehat{D}_\infty}\equiv L_{D_k}$. Since each element $x_k$ has spectrum $\bt$, then the quotient norm of $x_k$ is equal to $\|x_k\|$ and thus the sequence $\{y_k\}$ is unbounded.

We now consider the case of the crossed products. 
With the same notations as above, consider the following sequence
\begin{displaymath}
	x_k= U_{A^k e_1}
\end{displaymath}
Each $x_k\in \ca_k\subset \varinjlim \ca_i$.
We have that
$$
\Arrowvert [D_k,x_k] \Arrowvert =\| [M_{\ell_k}, U_{A^k e_1}]\| \leq  \sup_g |\ell_k(g) - \ell_k(g-A^k e_1)| \leq  \Arrowvert A^{k} e_1 \Arrowvert \leq \|A^k\| \to 0.
$$
Since ${\rm sp}(x_k)=\bt$, again the sequence $y_k:=x_k/ \|[D_k,x_k]\|$ has constant seminorm $L_{\widehat{D}_\infty}\equiv L_{D_k}$ and increasing quotient norm.

Finally we take care of the UHF-algebra.
Consider any  matrix $b\in \big(M_\molt(\mathbb{C}) \setminus \mathbb{C}I\big)\subset UHF(\molt^\infty)$. 
 We define the following sequence
\begin{displaymath}
	x_n = I_{[-\infty,-n-1]} \otimes b\otimes I_{[-n+1,+\infty]},
\end{displaymath}
where with the above symbol we mean that the matrix $b$  is in the position  $-n$ inside an infinite bilateral product where each factor is labelled  by an integer. 
A quick computation shows that
$$
[Q_k,x_n]=\left\{\begin{array}{ll}
 		0 & \textrm{ if $k>-n$}\\
		{\rm id}_{-\infty, k-1}\otimes (b {\rm Tr}(\cdot)-{\rm Tr}(b\cdot ))\otimes \t & \textrm{ if $k=-n$} \\
		{\rm id}_{-\infty, k-1}\otimes F\otimes \big(\bigotimes_{i=k+1}^{-n-1}{\rm Tr}(\cdot)\big)\otimes({\rm Tr}(b\cdot )-b {\rm Tr}(\cdot))\otimes \t & \textrm{ if $k<-n$.}\\
 		\end{array}
 		\right.	
$$ 
This means that $[D_\infty, x_n]=\sum_{k\leq -n}r^{ks}[Q_k, x_n]$. 

We observe that each $x_n$ has non-zero seminorm. In fact,
\begin{eqnarray*}
		\Arrowvert[D_\infty, x_n]\Arrowvert &= & \sup_{\|\xi\|=1}\Arrowvert[D_\infty, x_n]\xi\Arrowvert\\
	&\geq &  \Arrowvert[D_\infty, x_n]x_n^*\Arrowvert\\
	&=& r^{-ns} \| {\rm Tr}(bb^*)-b {\rm Tr}(b^*)\|>0
\end{eqnarray*}
where in the last line we used that $[Q_k,x_n]x_n^*=0$ for all $k\neq -n$.
Moreover, we have that
$$
\|[D_\infty, x_n]\|\leq 2\|b\| \left(\sum_{k\leq -n}r^{ks}\right)=2\|b\|   \frac{r^{s-ns}}{1-r^s}
$$
which tends to zero as $n$ goes to infinity.

The sequence $y_k:=x_k/ \|[D_\infty,x_k]\|$ has bounded seminorm $L_{\widehat{D}_\infty}$   
and increasing quotient norm. 

We end this proof with an explanation of the second part of the statement of this Theorem. First of all, we observe that if the sequence of the normic radii of the balls $B_{L_{D_n}}$  diverges, then $B_{L_{\widehat{D}_\infty}}$ contains an unbounded subset with unbounded quotient norm. Therefore, since a compact subset is bounded, the ball $B_{L_{\widehat{D}_\infty}}$ cannot be compact.
\end{proof}

In a recent paper (\cite{LP})   Latr\'emoli\`ere and Packer studied the metric structure of noncommutative solenoids, namely of the inductive limits of quantum tori. In particular, they considered noncommutative tori as quantum compact  metric spaces and proved that their inductive limits, seen as quantum compact metric spaces, are also  limits in the sense of Gromov-Hausdorff propinquity (hence quantum Gromov-Hausdorff) of the inductive families. In our setting the inductive limit  of the quantum tori is no longer a quantum compact  metric space. The different result is a consequence of the different metric structure considered.  Latr\'emoli\`ere and Packer described the inductive limit as a  twisted group $C^*$-algebra on which there is an ergodic action of $G_\infty:=\varprojlim \bt$, and according to  Rieffel (\cite{Rief}) a continuous length function on $G_\infty$ gives rise to a Lip-seminorm. In our setting the seminorm may  also be described in the same way, however the corresponding length  function is unbounded, thus not continuous. We give an explicit description of this situation in a particular example. 

\begin{example}
Consider the two-dimensional rational rotation algebra $A_\theta$, with $\theta=1/3$. With the former notation, set
$$
B=\left( \begin{array}{cc}
2 & 0\\
0 & 2
\end{array} \right),
$$
and define the morphism $\alpha: A_\theta \to A_\theta$ by $\alpha(U)=U^2$, $\alpha(V)=V^2$. Now we may consider the inductive limit $\varinjlim \cb_n$ where $\cb_n=A_\theta$ (see \eqref{ind-lim-diag}). We observe that this case also fits in the setting of  Latr\'emoli\`ere and Packer (see  \cite[Theorem 3.3]{LP}). Then, there exists a length function  that induces the seminorm $L_{\widehat{D}_\infty}$.
\end{example}
\begin{proof}
Consider the standard length function on the circle $\zcl(e^{2\pi i t}):=|t|$ for $t\in (-1/2,1/2]$. There is an induced length function on $\bt^2$, namely $\ell_0(z_1,z_2):=\max\{\zcl(z_1),\zcl(z_2)\}$.
We define the following length function $\ell(g):=\sup_n 2^n\ell_0(g_n)$ on the direct product $\prod \bt^2$, thus by restriction also on the projective limit $G_\infty:=\varprojlim \bt$ (with respect to the  projection $\pi\equiv \alpha^*: \bt^2 \to \bt^2$, $\pi((z_1,z_2))=(2z_1,2z_2)$). 
For any $\phi\in \mathbb{R}^2$ we define the following action on $A_\theta$: $\widetilde{\rho}_\phi(f)(t):=f(t+3\phi)$. 
 Since $\theta=1/3$, $\widetilde{\rho}$ is the identity on $A_\theta$ when $\phi\in\bz^2$, hence there is an induced action of $\bt^2=\mathbb{R}^2/\bz^2$ on $A_\theta$. We denote this action with $\rho$. 
There is a naturally induced action $\rho^\infty$ of the group $\prod_{i=0}^\infty \bt^2$ on $\prod_{i=0}^\infty A_\theta$ given by $\rho^\infty_g(f_0,f_1,\ldots):=(\rho_{g_0}(f_0),\rho_{g_1}(f_1),\ldots)$ for any $g\in \prod_{i=0}^\infty \bt^2$ and any $(f_0,f_1,\ldots)\in \prod_{i=0}^\infty A_\theta$. We now check that the restriction of this action to $G_\infty$ gives rise to an action on $\varinjlim A_\theta$. It is enough to prove the claim on the algebraic inductive limit ${\rm alg-}\varinjlim A_\theta$. Let $(f_0,f_1,\ldots)\in {\rm alg-}\varinjlim A_\theta$. By definition there exists  $n\in \mathbb{N}$ such that $f_{n+i}=\alpha^i(f_n)$ for all $i\in \mathbb{N}$. For any $g\in G_\infty$, we have that
$$
\rho_{g_{n+i}}(f_{n+i})=\rho_{g_{n+i}}(\alpha^i (f_n))=\alpha^i (\rho_{\pi^n(g_{n+i})}(f_n))=\alpha^i (\rho_{g_n}(f_n)).
$$

For any $g\in G_\infty$ and any $X\in \varinjlim A_\theta$ we define the following seminorm
$$
L_{\rho^\infty,\ell}(X):=\sup_{g\in G_\infty}  \frac{\|\rho^\infty_g(X)-X\|}{\ell(g)}.
$$
Any element $f_n\in \cb_n$ embeds into $\varinjlim A_\theta$ as $X=(\underbrace{0,\ldots , 0}_n,f_n, \alpha(f_n), \alpha^2(f_n)\ldots)$. We have that
\begin{align*}
L_{\rho^\infty,\ell}(X) & =\sup_{g\in G_\infty}  \frac{\|\rho^\infty_g(X)-X\|}{\ell(g)}\\
& = \sup_{g\in G_\infty}  \limsup_i \frac{\| \alpha^i(f_n)(z+3g_{n+i})- \alpha^i(f)(z)\|}{\ell(g)}\\
& = \sup_{g\in G_\infty}  \frac{\|f_n(z+3g_{n})- f_n ( z)\|}{\ell_0(g_n)} \frac{\ell_0(g_n)}{\ell(g)}\\
& =  \left(\sup_{g_n} \frac{\|f_n(z+3g_{n})- f_n ( z)\|}{\ell_0(g_n)}\right)  \left( \sup_{g\in G_\infty} \frac{\ell_0(g_n)}{\ell(g)}\right) \\
& = \frac{L_0(f)}{2^n},
\end{align*}
where the last two equalities  hold because, for any $g_n\in \bt^2$, we may find a sequence $g=\{g_i\}$ such that $\ell(g)=2^n\ell_0(g_n)$ (if $g_n=e^{2\pi i t}$ for $t\in(-1/2,1/2]$ consider $g_{n+k}=e^{2\pi i t /2^{k}}$) and
$L_0$ is the Lipschitz seminorm $\sup_{h\in \bt^2} \frac{\|f(z+h)- f (z)\|}{\ell_0(h)}$, which is equivalent to $L_{D_0}$ (see \cite{Rief}).
Denote by $\varphi_n: \cb_n=A_\theta\to \ca_n$ the natural isomorphism  given by $\varphi_n(W(m,t)):=e^{2\pi i \theta (2^{-n} m,t)} W_0(2^n m)$ (cf. \eqref{UBVB}) and  consider the following seminorm on $\cb_n$: $L_n(x):=L_{D}(\varphi_n(x))=\|[D_n,\varphi_n(x)]\|$. 
Since the seminorm $L_{D}$ is expressed in terms of the norm of some linear combinations of the two derivatives,  one has that $L_n(x)=2^{-n}L_0(x)$. Therefore, the former computation leads to $L_{\rho^\infty,\ell}=L_n$, when restricted to $\cb_n$.
\end{proof}

\appendix
\section{Some results in noncommutative integration theory}

Let $(\cam,\t)$ be a von Neumann algebra with a f.n.s. trace,  $T\widehat{\in}\cam$  a self-adjoint operator. 
We use the notation $e_T(\O)$ for the spectral projection of  $T$ relative to the measurable set $\O\subset \br$, and $\l_T(t):=\t(e_{|T|}[t,+\infty))$, $\mu_T(t):=\inf \{s: \lambda_T(s)\leq t\}$, for a $\t$-compact operator $T$.

\begin{lemma}\label{varie}
Let $(\cam,\t)$ be a von Neumann algebra with a f.n.s. trace,  $T\widehat{\in}\cam$  a self-adjoint operator, such that $\Lambda _{T}(s):=\t(e_{T}(-s,s))<\infty$ for any $s>0$. Then
\item[$(1)$] $\Lambda _{T}(s)=\sup\{\t(e):\|Te\|<s, e\in\text{Proj}(\cam)\}$, $s>0$,
\item[$(2)$] if $C\in\cam_{sa}$, and $c:=\|C\|$, then $\t(e_{T+C}(-s,s))<\infty$ for any $s\geq0$, and 
$\Lambda _{T+C}(s)\leq\Lambda_{T}(s+c)$,

\item[$(3)$] if $e_{T}(\{0\})=0$, $T^{-1}$ is $\t$-compact and $\Lambda _{T}(s)=\l_{|T|^{-1}}(s^{-1})$, $s>0$.
\end{lemma}
\begin{proof}
$(1)$ Indeed,
$$
a:=\t(e_{T}(-s,s))=\sup\{\t(e_{T}(-\s,\s)):0\leq\s<s\}\leq\sup\{\t(e):\|Te\|<s\}.
$$
Assume, by contradiction, there is $e\in\text{Proj}(\cam)$ such that $\t(e)>a$ and $\|Te\|<s$. For $\xi\in e\ch\cap e_{|T|}[s,\infty)\ch$, $\norm{\xi}=1$, we have $(\xi,T^*T\xi)<s^2$ and $(\xi,T^*T\xi)\geq s^2$, namely $ e\wedge e_{|T|}[s,\infty)=\{0\}$. As a consequence,
$$
e_{|T|}[s,\infty)=e_{|T|}[s,\infty) - e\wedge e_{|T|}[s,\infty) \sim e\vee  e_{|T|}[s,\infty) -e\leq I-e
$$
where $\sim$ stands for Murray - von Neumann equivalence. Passing to the orthogonal complements we get $a= \t(e_{T}(-s,s))\geq \t(e)>a$, which is absurd.

\item[$(2)$] Set $\Omega_{T,s}=\{e\in\text{Proj}(\cam):\|Te\|<s\}$; since $\|Te\|\leq\|(T+C)e\|+c$,  we have that $\Omega_{T+C,s}\subseteq\Omega_{T,s+c}$, . The thesis follows from $(1)$.

\item[$(3)$] A straightforward computation shows that $e_{|T|^{-1}}(s,+\infty)=e_T(-1/s,1/s)$. Therefore $T^{-1}$ is $\t$-compact  \cite{FaKo} and the equality follows.
\end{proof}

\begin{lemma}\label{sameResA}
Let $(\cam,\t)$ be a von Neumann algebra with a f.n.s. trace,  $T\widehat{\in}\cam$  a positive self-adjoint operator $T$, with $\t$-compact resolvent,  $\som,t>0$. Then, the following are equivalent

\item[$(1)$] exists $\res_{s=\som} \ \t(T^{-s}e_T[t,+\infty))=\a\in\br$, 

\item[$(2)$] exists $\res_{s=\som} \ \t((T^2+1)^{-s/2})=\a\in\br$.
\end{lemma}
\begin{proof}
Let us first observe that
\begin{align}
\t(T^{-s}e_T[t,+\infty))&=\int_t^{\infty}\l^{-s}d\t(e_T(0,\l)),\label{integr-1}
\\
\t((T^2+1)^{-s/2})& = \int_0^{\infty}(\l^2+1)^{-s/2}d\t(e_T(0,\l)),\label{integr-2}
\end{align}
and
\begin{align*}
	(t^2+1)^{-s/2} & \leq(\l^2+1)^{-s/2} \leq 1, \quad \forall \l\in [0,t], \\
 t^s(1+t^2)^{-s/2} \l^{-s} & \leq (\l^2+1)^{-s/2}  \leq \l^{-s} , \quad \forall \l\in [t,+\infty), 
\end{align*}
therefore the finiteness of any of the two residues in the statement implies the finiteness of the two integrals \eqref{integr-1}, \eqref{integr-2} above for any $s>\som$. Then,
\begin{align*}
|\t(T^{-s}e_T[t,+\infty))&-\t((T^2+1)^{-s/2})|
= \bigg| \int_t^{\infty}\l^{-s}d\t(e_T(0,\l))-\int_0^{\infty}(\l^2+1)^{-s/2}d\t(e_T(0,\l)) \bigg| \\
&\leq\int_0^{t}(\l^2+1)^{-s/2}d\t(e_T(0,\l))
+\frac{s}2\int_t^\infty\l^{-s-2}d\t(e_T(0,\l)),
\end{align*}
where the inequality follows by
$$
\l^{-s}-(\l^2+1)^{-s/2}=\l^{-s}[1-(1+\frac1{\l^2})^{-s/2}]\leq\frac{s}2\l^{-s-2},
$$
which, in turn, follows by 
$$
g(x)=1-(1+x)^{-s/2}\leq\sup_{\xi\in[0,x]}g'(\xi)\,x=\frac s2 x,\quad \text{for }x\geq 0
$$
Finally, taking the limit for $s\to \som^+$, we get
$$
\lim_{s\to \som^+}|\t(T^{-s}e_T[t,+\infty))-\t((T^2+1)^{-s/2})|
\leq\t(e_T(0,t))+\frac{\som}2\int_t^\infty\l^{-(\som+2)}d\t(e_T(0,\l))<\infty,
$$
where the last integral is \eqref{integr-1} with $s=d+2$, hence is finite, and we have proven 
 the thesis.
\end{proof}

\begin{lemma}\label{sameResB}
Let $(\cam,\t)$ be a von Neumann algebra with a f.s.n. trace, $T$ a self-adjoint  operator affiliated with $\cam$ with bounded compact inverse, $C\in\cam_{sa}$  such that $T+C$ has bounded inverse. Then, the following are equivalent

\item[$(1)$] exists $\res_{s=\som}\ \t(|T|^{-s})=\a\in\br$,

\item[$(2)$] exists $\res_{s=\som}\ \t(|T+C|^{-s})=\a\in\br$.
\end{lemma}
\begin{proof} 
It is enough to prove that $(1)\implies (2)$. Set $c:=\| C\|$. From Lemma \ref{varie},
we get $\Lambda _{T+C}(s)\leq\Lambda_{T}(s+c)$ for every $s>0$, hence $\l_{|T+C|^{-1}}(s)\leq\l_{|T|^{-1}}(\frac{s}{1+cs})$. Then, for $0<\th<1$, 
\begin{align*}
\m_{|T+C|^{-1}}(t) & = \inf \{ s\geq0 : \l_{|T+C|^{-1}}(s)\leq t \} \\
& \leq \inf \{ s\geq0 : \l_{|T|^{-1}}(\frac{s}{1+cs})\leq t \}\\
&=\inf\{\frac{h}{1-ch} \geq0 : \l_{|T|^{-1}}(h)\leq t\}\\
&=\inf\{\frac{h}{1-ch} : 0\leq h <c^{-1}, \l_{|T|^{-1}}(h)\leq t\}\\
&\leq\inf\{\frac{h}{1-ch} : 0\leq h \leq\th c^{-1}, \l_{|T|^{-1}}(h)\leq t\}\\
&\leq(1-\th)^{-1}\inf\{h : 0\leq h \leq\th c^{-1}, \l_{|T|^{-1}}(h)\leq t\}\\
&=\begin{cases}
(1-\th)^{-1}\inf\{h\geq0: \l_{|T|^{-1}}(h)\leq t\}, & \text{if }\l_{|T|^{-1}}(c^{-1}\th)\leq t,\\
+\infty, &\text{otherwise,}
\end{cases}
\\
&=\begin{cases}
(1-\th)^{-1}\m_{|T|^{-1}}(t), & \text{if }\l_{|T|^{-1}}(c^{-1}\th)\leq t,\\
+\infty, & \text{otherwise.}
\end{cases}
\end{align*}
As a consequence,
\begin{align*}
\t(|T+C|^{-s})&=\int_0^\infty\m_{|T+C|^{-1}}(t)^s\,dt \\
& \leq \int_0^{\l_{|T|^{-1}}(c^{-1}\th)} \m_{|T+C|^{-1}}(t)^s \, dt + \int_{\l_{|T|^{-1}}(c^{-1}\th)}^{+\infty} (1-\th)^{-s} \m_{|T|^{-1}}(t)^s\, dt \\
&= \int_0^{\l_{|T|^{-1}}(c^{-1}\th)}\left(\m_{|T+C|^{-1}}(t)^s-(1-\th)^{-s} \m_{|T|^{-1}}(t)^s  \right)dt
+(1-\th)^{-s}\t(|T|^{-s})\\
&\leq\left(\|(T+C)^{-1}\|^s+(1-\th)^{-s}\|T^{-1}\|^s\right)\l_{T^{-1}}(c^{-1}\th)+(1-\th)^{-s}\t(|T|^{-s})
<\infty.
\end{align*}
Passing to the residues, we get $\limsup_{s\to \som^+}(s-\som)\t(|T+C|^{-s})\leq (1-\th)^{-\som} \res_{s=\som}\ \t(|T|^{-s})$, hence, by the arbitrariness of $\th$, 
$\limsup_{s\to \som^+}(s-p)\t(|T+C|^{-s})\leq \res_{s=\som}\ \t(|T|^{-s})$. Exchanging $T$ with $T+C$ we get $\res_{s=\som}\ \t(|T|^{-s})\leq \liminf_{s\to \som_+}(s-p)\t(|T+C|^{-s})$, hence the thesis. 
\end{proof}

\begin{proposition}\label{sameRes}
Let $(\cam,\t)$ be a von Neumann algebra with a f.n.s. trace, $T$ a self-adjoint  operator affiliated with $\cam$ with compact resolvent, $C\in\cam_{sa}$. Then, the following are equivalent

\item[$(1)$] exists $\res_{s=\som}\  \t((T^2+1)^{-s/2})=\a\in\br$,

\item[$(2)$] exists $\res_{s=\som}\ \t((T+C)^2+1)^{-s/2})=\a\in\br$. 

In particular, the abscissas of convergence coincide.
\end{proposition}
\begin{proof}
By Lemma \ref{sameResA}, the thesis may be rewritten as 
$$
\exists\ \res_{s=\som}\ \t(|T|^{-s}e_{|T|}[t,+\infty)) = \a\in\br \Longleftrightarrow \exists \ \res_{s=\som}\ \t(|T+C|^{-s}e_{|T+C|}[t,+\infty))=\a.
$$
Since the operator
\begin{align*}
C' := &(T+C)e_{|T+C|}[t,+\infty) - Te_{|T|}[t,+\infty) \\
& = (T+C)e_{|T+C|}[0,+\infty) - Te_{|T|}[0,+\infty) - (T+C)e_{|T+C|}[0,t) + Te_{|T|}[0,t) \\
& = C-(T+C)e_{|T+C|}[0,t) +Te_{|T|}[0,t)
\end{align*}
is bounded and self-adjoint, we may apply Lemma \ref{sameResB} to the operators
$(T+C)e_{|T+C|}[t,+\infty)$ and $Te_{|T|}[t,+\infty)$, proving the Proposition. 
\end{proof}

\emph{Acknowledgement}
This work was supported by the following institutions:
the ERC for the Advanced Grant 227458 OACFT ``Operator Algebras and Conformal Field Theory'', the MIUR PRIN ``Operator Algebras, Noncommutative Geometry and Applications'', the INdAM-CNRS GREFI GENCO, and the INdAM GNAMPA.

\end{document}